\documentclass[11pt]{article}
\usepackage{amsmath,amssymb,enumerate,bbm}

\newcommand{\op}[1]{#1}
\newcommand{\tmem}[1]{{\em #1\/}}
\newcommand{\tmmathbf}[1]{\ensuremath{\boldsymbol{#1}}}
\newcommand{\tmname}[1]{\textsc{#1}}
\newcommand{\tmop}[1]{#1} 
\newcommand{\tmstrong}[1]{\textbf{#1}}
\newcommand{\tmtextbf}[1]{{\bfseries{#1}}}
\newcommand{\tmtextit}[1]{{\itshape{#1}}}
\newcommand{\tmtextrm}[1]{#1}
\newcommand{\tmtextsc}[1]{{\scshape{#1}}}
\newcommand{\tmtexttt}[1]{{\ttfamily{#1}}}
\newcommand{\tmtextup}[1]{{\upshape{#1}}}
\newenvironment{enumeratealpha}{\begin{enumerate}[a{\textup{)}}] }{\end{enumerate}}
\newenvironment{enumeratenumeric}{\begin{enumerate}[1.] }{\end{enumerate}}
\newenvironment{proof}{\noindent\textbf{Proof\ }}{\hspace*{\fill}$\Box$\medskip}
\newtheorem{theorem}{Theorem}[section]
\newtheorem{corollary}[theorem]{Corollary}
\newtheorem{definition}[theorem]{Definition}
\newtheorem{lemma}[theorem]{Lemma}
\newtheorem{varremark}[theorem]{Remark}

\def\ed{\end{document}}
\def\LBox{\large{\Box}}
\def\nod{\noindent}
\def\bm{\boldmath}
\def\nb{\bar{\nu}}
\begin{document}

\title{Global Square and Mutual Stationarity at the $\aleph_n$}\author{Peter Koepke \\
Dept. of Mathematics,\\
University of Bonn,\\
Beringstrasse 1,\\
D-53115 Bonn,\\
Germany\and Philip Welch\thanks{The second author would like to
express his gratitude to the Deutsche Forschungsgemeinschaft \ for
the support of a Mercator Gastprofessur and to the Mathematics
Department of the University of Bonn where it was held.}\\ School of
Mathematics,\\ University of Bristol,\\Bristol, BS8 1TW\\England}
\maketitle

\begin{abstract}
  We give a proof of a theorem of \tmtextsc{Jensen and Zeman} on the existence
  of Global {\large$\Box$} in the Core Model below a measurable cardinal $\kappa$
  of Mitchell order
  (``$o_M (\kappa)$'') equal to $ \kappa^{+ +}$, and use it to prove the following theorem on
  mutual stationarity at the $\text{$\aleph_n$}$.

  Let $\omega_1$ denote the first uncountable cardinal of $V$  and
  set $\mathrm{Cof}(\omega_1)$ to be the class of ordinals of cofinality
  $\omega_1$.

  \tmtextbf{Theorem: }{\tmem{If every sequence $(S_n)_{n < \omega}$ of
  stationary sets $S_n \subseteq \mathrm{Cof}(\omega_1) \cap \aleph_{n + 2}$,
  is mutually stationary, then there is an inner model with infinitely many
  inaccessibles $(\kappa_n)_{n < \omega}$ so that for every $m$ the class of
  measurables $\lambda$ with $o_M(\lambda) \geq\kappa_m$ is stationary in $\kappa_n$ for all
  \ $n > m.$ In particular, there is such a model in which for all
  sufficiently large $m < \omega$ the class of measurables $\lambda$
  with
  $o_M(\lambda) \geq \omega_{m}$ is, in $V$, stationary below $\aleph_{m + 2}$.}}

\end{abstract}

\section{Introduction}

This paper extends previous investigations into the nature of
\tmtextit{mutual stationarity}, a concept introduced by {\tmname{M.
Foreman}} and {\tmname{M. Magidor}} {\cite{Foreman-Magidor2002}} in
order to transfer some combinatorial aspects of stationary subsets
of regular cardinals to singular cardinals. They made particular use
of this in investigating the non-saturation of the non-stationary
ideals of the form $\mathcal{P}_{\kappa} (\lambda)$.

Our purpose here is to establish that the mutual stationarity
property at $\aleph_{\omega}$ (or more precisely at the sequence of
the first $\omega$-many uncountable cardinals, $\langle \aleph_n
\mid 0 < n < \omega \rangle$), is a {\tmem{large cardinal}}
property, that is, it entails the consistency of {\tmem{strong
axioms of infinity}} which concern measurable cardinals. The
definition of mutual stationarity is more general than this however:

\begin{definition}
  Let $(\kappa_n)_{n < \omega}$ be a strictly increasing sequence of regular
  cardinals $\geqslant \aleph_2$ with $\kappa_{\omega} =$sup$_{n < \omega}
  \kappa_n$. A sequence $(S_n)_{n < \omega}$ is called \tmtextup{mutually
  stationary in} $(\kappa_n)_{n < \omega}$ if every first-order structure
  $\mathfrak{A}$ of countable type with $\kappa_{\omega} \subseteq
  \mathfrak{A}$ has an elementary substructure $\mathfrak{B} \prec
  \mathfrak{A}$ such that  $$\forall n < \omega \sup |\mathfrak{B}| \cap
  \kappa_n \in S_n .$$
\end{definition}

{\tmname{M. Foreman}} and {\tmname{M. Magidor,}} together with
{\tmname{J. Cummings}} further investigated the status of such
sequences in {\cite{CFM06}}. Note that if $(S_n)_{n < \omega}$ is
mutually stationary in $(\kappa_n)_{n < \omega}$ then each $S_n \cap
\kappa_n$ is stationary in $\kappa_n$. In the following we shall
denote the class $\{\xi \in \tmop{Ord} \mid \tmop{cf} (\xi) =
\lambda\}$ by Cof$(\lambda)$.

\begin{definition}
  Let $(\kappa_n)_{n < \omega}$ be a strictly increasing sequence of regular
  cardinals and $\lambda < \kappa_0$, $\lambda$ regular. The
  {\tmem{{\tmem{\tmtextup{mutual stationarity property}}}}}
  $\tmtextrm{\tmop{MS}} ((\kappa_n)_{n < \omega}, \lambda)$ is the statement:
  if $(S_n)_{n < \omega}$ is a sequence of stationary sets $S_n \subseteq
  \tmtextrm{\tmop{Cof}}(\lambda) \cap \kappa_n$, \ then $(S_n)_{n < \omega}$
  is{\tmem{ {\tmem{mutually stationary}}}} in $(\kappa_n)_{n <
  \omega}$.{\tmname{}}
\end{definition}

{\tmname{M. Foreman}} and {\tmname{M. Magidor}}
{\cite{Foreman-Magidor2002}} proved the following two theorems:\\

{\noindent}{\tmstrong{Theorem.}} {\tmem{For $(\kappa_n)_{n <
\omega}$ be any strictly increasing sequence of uncountable regular
cardinals:}\\ (i) $\tmtextrm{\tmop{MS}} ((\kappa_n)_{n < \omega},
\omega)$
\tmem{holds.}} \\
\nod (ii) $\tmtextrm{\tmop{MS}} ((\kappa_n)_{n < \omega},
\omega_1)${\tmem{ implies}} $V \neq L$.\\

This did not yet say that $\tmop{MS}$ was a large cardinal property.
That it was is the left to right direction of the following
equivalence, proven in \cite{KW05}:

\begin{theorem}
  \label{Equiconsistency}The theories $\tmtextrm{\tmop{ZFC}} + \exists
  (\kappa_n)_{n < \omega}  \tmtextrm{\tmop{MS}} ((\kappa_n)_{n < \omega},
  \omega_1)$ and $\tmtextrm{\tmop{ZFC}} + \exists \kappa( \kappa$
  measurable$)$
  are equiconsistent.
\end{theorem}

The implication from right to left was first proven by
{\tmname{Cummings}}, {\tmname{Foreman}}, and {\tmname{Magidor}}
{\cite{CFM01}} {\tmem{via}} Prikry forcing. They proved more than
this: they showed that a tail of the Prikry generic sequence
satisfies $\tmtextrm{\tmop{MS}} ((\kappa_n)_{n < \omega}, \lambda)$
for any $\lambda < \kappa_0$ (or indeed the mutual stationarity of
{\tmem{any}} sequence of stationary sets $S_n \subseteq \kappa_n$
irrespective of the cofinalities of the ordinals in the $S_n$). This
is essentially obtained by utilising the fact that a tail of the
Prikry generic sequence remains {\tmem{coherently Ramsey}} in the
generic extension. The forward direction was proven in \cite{KW05}
using the core model $K$ of {\tmname{A. J. Dodd}} and {\tmname{R. B.
Jensen}} (see {\cite{DoJe81}}). The deduction of the existence of
$0^{\sharp}$ from $\tmtextrm{\tmop{MS}} ((\kappa_n)_{n < \omega},
\omega_1)$ was done in detail, and the extension to proving the
existence of the inner model with a measurable was sketched, using
the hyperfine structure of {\tmname{S.Friedman}} and the first
author ({\cite{FK97}}). The proof involved the global square
principle $\square$ in $L$ and techniques from the Jensen Covering
theorem for $L$ (see {\cite{Devlin-Jensen}}). The purpose of this
paper is to give a full account of the interaction of the proof of
global $\LBox$ with the MS property, (insofar as we are able) thus
filling in the details of the above argument, but significantly
strengthening the result to obtain models with many measures of high
Mitchell order, in the case $(\kappa_n)_{n < \omega}$ consists of
consecutive sequences of cardinals mentioned in the abstract:

\begin{theorem}
  \label{Main}If $\tmtextrm{\tmop{MS}} ((\aleph_{n })_{1<n <
  \omega}, \omega_1)$ holds then there is an inner model, $K$, and there is $2< k
  < \omega$ so that for any $n$ with $k < n < \omega$ each $\aleph_{n}$ is a
  Mahlo limit (in $V$) of ordinals $\kappa$ which are, in $K$, measurable of
  Mitchell order $o_M (\kappa) = \omega_{n-2 }$. In fact, for such $\aleph_{n}$
  the ordinals $\alpha \in {\mathrm{Cof}}(\omega_{n -2})$ which are singular in
  $K$ are, in $V$, non-stationary below $\aleph_{n}$.
\end{theorem}

One might wonder whether increasing the cofinality of the independently chosen
stationary sets might yield increased Mitchell order. Well, perhaps, but
seemingly not by our methods. The following is a corollary to the proof of the
above theorem.

\begin{corollary}
  Let $m$ be fixed, $1 \leq m < \omega$. Then if $\tmtextrm{\tmop{MS}}
  ((\aleph_{n + m })_{0<n < \omega}, \omega_m)$ holds, exactly the same
  conclusion as that of Theorem \ref{Main} may be drawn.
\end{corollary}

The methods here seem just short of allowing us to conclude that there is an
inner model with a measurable $\kappa$ with Mitchell order of $\kappa$equal to
$\kappa :$ (``$o_M (\kappa) = \kappa$'').

It is important in the above statement that we use all the alephs
below $\aleph_{\omega}$ (from some point on) since the first author
has shown that omitting a cardinal above each one for which we wish
to consider arbitrary stationary sets, has a much weaker consistency
strength, (see {\cite{Ko07}}).

\begin{theorem}
  The theories $\tmtextrm{\tmop{ZFC}} + \tmtextrm{\tmop{MS}} ((\aleph_{2 n +
  1})_{n < \omega}, \omega_1)$ and $\tmtextrm{\tmop{ZFC}} + \exists \kappa
  (\kappa$ a measurable cardinal$)$ are equiconsistent.
\end{theorem}

It is unknown whether $\tmtextrm{\tmop{MS}} ((\aleph_{n })_{1< n <
\omega}, \omega_k)$ (for any $k\geq 1$), when taking all the
cardinals from some point on, is consistent relative to any large
cardinals. The model $K$ in Theorem \ref{Main} can be taken to be
the core model built using measures (partial or full) only on its
constructing extender sequence.

We shall need the following formulation of the Weak Covering Lemma
due to \tmtextsc{W.Mitchell} (\tmtextit{cf.} {\cite{Mihbk}})

\begin{theorem}\label{WCL}{\bf  (Weak Covering Lemma)}
  {\tmem{{\tmem{Assume there is no inner model with a measurable cardinal
  $\kappa$ with $o_M (\kappa) = \kappa^{+ +} .$ Let $\alpha$ be regular in $K$
  with $\omega_1 \leq \gamma = \tmop{cf} (\alpha) < \tmop{card} (\alpha)$.
  Then in $K$ we have $o_M (\alpha) \geq \gamma$.}}}}
\end{theorem}

We shall assume a development of the fine structure of such a core
model $K$, as can be found in {\tmname{M. Zeman}} {\cite{Z02}}. \
$K$ is thus a model of the form $L [E]$ with $E$ a sequence of
partial or full extenders in the manner of Zeman's book. However no
such extender requires any generator beyond that of its critical
point. We shall need to consider the proof of the existence of
global square $\LBox$ in such a model. This is known to hold,{\tmem{
cf}} {\cite{JeZe00}}. \ The fine structural notation we shall adopt
is that of the book (which is also that of the paper cited). The
indexing of extenders will be the Friedman-Jensen indexing whereby
an extender is placed on the $E$ sequence of a hierarchy at
precisely the successor cardinal of the image of the critical point
by that extender. Again this is following {\cite{JeZe00}}.

Jensen and Zeman's method of proof for global $\LBox$ is to define a
``smooth category'' of structures and maps from which it is known
that a global $\LBox$ sequence can be derived. This latter
{\tmem{derivation}} is purely combinatorial and so requires no
inspection of the fine structure of the original model. The burden
of their proof is the {\tmem{construction }}of the smooth category
itself. However that construction does not yield an explicit
computation for the order types of the various $C_{\nu}$ sequences.
(It is the latter derivation that does that). For our proof we need
to have a construction of global $\LBox$ where we can see (i) what
those order types will be and how they are arrived at; and (ii) that
order types for certain $C_{\nu}$-like sequences will (on a tail)
not be prolonged by iterations of the mouse from which they are
defined. We give a proof of Global $\LBox$ {\tmem{ab initio}}
directly without going through the smooth category. This is done in
Section 3. In section 2 we give some  fine structural lemmas that
form the hard work of Jensen and Zeman's account in {\cite{JeZe00}}
which establish the right forms of parameter preservation and
appropriate condensation lemmata.  We merely quote these as
Condensation Lemmas (I) and (II).  However in order to prove that
the order types of $C_{\nu}$ sequences are not prolonged by
iterations of the structure over which they are defined we need to
prove the preservation of the $d$-parameters of {\cite{JeZe00}}.
This is at Lemma \ref{dpreserve}.  The analysis of the Condensation
Lemmata apart,  we try to keep the rest of the proof as
self-contained as possible. The proofs of Lemmas \ref{4.3} and
\ref{initseg} in particular repeat the proofs of \cite{JeZe00} 4.3
and 4.5. These are key lemmata on the relationships between
singularising structures and the maps between them, and are, in the
$\Sigma^\ast$ terminology, the successors to \cite{BJW} Lemmas 6.15
and 6.18. From Definition \ref{Bplus} onwards this is an account
very much following that of \cite{BJW} (and which will be in the
forthcoming \cite{Wehbk}), but modestly dressed in the appropriate
$J_s$ mouse notation. In Section 4 we see how to use features of
this proof to get the main Theorem \ref{Main}; the reader who is
completely familiar with the $\square$ proof and wants to discover
the ideas in the application to mutual stationarity may wish to go
straight there.


\section{Fine structural prerequisites}

For an acceptable $J$-structure $M$ we assume familiarity with the
notions of the uniformly defined $\Sigma_1$-Skolem function for $M$,
$h_M$ ,and of the class of parameter sequences $\Gamma^M$, and the
parameter sets $P^n_M, P_M, P_M^{\ast,}, R^n_M, R_M$, and
$R^{\ast}_M$. We shall write $\rho_M$ as usual for the
$\Sigma_1$-projectum of $M$. Similarly we shall write for the
\tmtextit{$n$+1'st projectum } $\rho^{n + 1}_M =_{\tmop{df}} \min
\{\rho_{M^{n, p}} \hspace{0.25em} | \hspace{0.25em} p \in \Gamma^n_M
\}$. We may assume that parameters are finite sets of ordinals. This
applies as well to the $n$\tmtextit{'th-standard parameter} and the
\tmtextit{standard parameter} denoted here $p^n_M, p_M$ respectively
for a structure $M$ as above. We wellorder $[\tmop{On}]^{< \omega} $
by $u <^{\ast} v \leftrightarrow \max (u \Delta v) \in v$. For $X
\subseteq \tmop{Ord}$ a set, we write $\tmop{ot} (X)$ for its order
type, and by $X^{\ast}$ we mean the set of limit points of $X$. Our
discussion of fine structure is entirely in the language of
$\Sigma_k^{(n)}$ relations due to Jensen (for which see \cite{Z02}
or \cite{Wehbk}). Boldface relations such as
$\tmmathbf{\Sigma}^{(n)}_1 (M)$ denote those definable using
parameters (in this case from $M$.)

\begin{definition}
  {\tmstrong{$(\Sigma_1^{(n)}$-{\em Skolem Functions)}}} Let $M$ be an acceptable
  $J$-structure, and let $p \in \Gamma^n_M$.

  (i) $h^{n, p}_M = h_{M^{n, p}}$;

  (ii) $\tilde{h}^n_M (w^n, x^0) = g_0 (g_1 \cdots g_{n - 1} ((w^n)_0,
  \langle (w^n)_1, x^0 (n - 1) \rangle) \cdots x^0 (0) \rangle)$ where, for $i
  \leq n$
  \[ g_i (\langle j, y^{i + 1} \rangle, p) = h_{M^{i, p \upharpoonright i}}
     (j, \langle y^{i + 1}, p (i) \rangle) . \]
\end{definition}

Then $g_i$ is uniformly lightface $\Sigma^{(i)}_1 (M)$ in the
variables shown. Thus $\tilde{h}^n_M$ is $\Sigma_1^{(n - 1)}$
uniformly over all $M$. The $\Sigma_1$ hull of a set $X \subseteq
M^{n, p}$ we shall denote by $h^{n, p}_M (X)$ (and is thus the set
$\{h^{n, p}_M (i, x)) \hspace{0.25em} | \hspace{0.25em} i \in
\omega, x \in X\}$). Note that $\tilde{h}^1_M (\langle j, y^0
\rangle, p (0)) = g_0 (j, \langle y, p (0) \rangle) = h_M (j,
\langle y, p (0) \rangle)$. If $p \in R^n_M$ then every $x \in M$ is
of the form $\tilde{h}^n_M (z, p)$ for some $z \in H^n_M$. We may
similarly form hulls using $\tilde{h}^n_M$: again if $X \subseteq
M^{n, p}$ say, and $q \in M$ then the $\Sigma_1^{(n - 1)}$ hull of
$X \cup \{q\}$ is the set $\{ \tilde{h}^n_M (x, q)) \hspace{0.25em}
| \hspace{0.25em} x \in X\}$ (we again may write $\tilde{h}^n_M (X
\cup \{q\})$ for this hull here). The following states some of these
facts and are easy to establish (see {\cite{Z02}} p.29):

\begin{lemma}
  \label{facts}Let $M$ be acceptable, and $p \in R^n_M$; {\tmem{{\tmem{\\
  (i{\tmem{) {\tmem{If $\omega \rho^n_M \in M$ and $p \in R^n_M$ then
  $\tilde{h}^n_M$ is a good, uniformly defined, $\Sigma_1^{(n - 1)} (M)$
  function mapping $\omega \rho^n_M$ onto}} $M.$}}\\
  (ii) (a) {\tmem{{\tmem{every $A \subseteq H^n_M$ which is}}
  $\tmmathbf{\Sigma}^{(n)}_1 (M)$ {\tmem{is}}}}}} }}{\tmem{}}
  $\tmmathbf{\Sigma_1} (M^{n, p})$;

  (b) $\rho^{n + 1}_M = \rho_{M^{n, p}}$
\end{lemma}

\begin{lemma}
  \label{Levy=*}Let $M$ be an acceptable $J$-structure. Then (i)
  $\tmmathbf{\Sigma^{\ast}} (M) \subseteq \text{$\tmmathbf{\Sigma}$}_{\omega}
  (M) .$ (ii) Let $p \in R^{\ast}_M$. Then
  ${\tmmathbf{\Sigma^{\ast}}}
  (M) = {\tmmathbf{\Sigma}}_{\omega} (M) .$
\end{lemma}

\begin{lemma}
  \label{sigma*}Let $\overline{M}, M$ be acceptable structures, and suppose
  $\pi : \overline{M} \longrightarrow M$ is $\Sigma_1^{(n)}$-preserving, and
  is such that $\pi \upharpoonright \omega \rho^{n + 1}_M = \tmop{id}$ and
  $\tmop{ran} (\pi) \cap P^{\ast}_M \neq \varnothing$. Then $\pi$ is $
  \Sigma^{\ast}$-preserving.
\end{lemma}

{\noindent}{\bf{Proof: }}This is {\cite{Z02}} 1.11.2. \mbox{}\hfill
Q.E.D. \\

Recall that a premouse $M$ is {\tmem{sound above $\nu$}} if $\omega
\rho^{n + 1}_M \leq \nu$ means that $\text{$\tilde{h}^{n +
1}_M$}$$(\nu \cup \{p_M \})= |M|$. We also say that it is
$k${\tmem{-sound}} if it is sound above $\omega \rho^k_M$.

In the next lemma there are various concepts that we shall quickly
gloss: $o^N (\kappa)$ is the extender order of $\kappa$ in the
hierarchies under consideration (and roughly corresponds to Mitchell
order of measures); the hat over a premouse, as in $\hat{N},$
indicates the {\tmem{expansion}} of the premouse structure $N$, to
which extenders are usually applied (as, for example, when
coiterations of premice are formed). The premice then act as
{\tmem{bookkeeping}} premice for the indices that are being used,
whilst the actual extenders are applied to these hatted expansions.
We simply follow the conventions of $\text{{\cite{JeZe00}}}$ and we
ignore the differences between these structures. The reader worried
about these details may consult $\text{{\cite{JeZe00}}}$ Sect. 2 or
Ch. 8 of {\cite{Z02}}.

\begin{theorem}
  {\tmem{\label{CLI}{\tmstrong{(Condensation Lemma I)}}}} ({\tmem{cf}}{\tmem{
  {\cite{JeZe00}} 2.1}}). Suppose there is no inner model for $o_M (\kappa) =
  \kappa^{+ +} .$ Let $N$ be a premouse, $M$ a mouse and $\sigma : \hat{N}
  \longrightarrow_{\Sigma_0^{(n)}} \hat{M}$ with $\sigma \upharpoonright
  \omega \rho^{n + 1}_N = \tmop{id} ;$ Then $N$ is a mouse; moreover if $N$ is
  sound above $\nu = \tmop{crit} (\sigma)$ then one of the following holds:

  (i) $N$ is the core of $M$ above $\nu$ and $\sigma$ is the iteration map,
  which is the corresponding core map;

  (ii) $N$ is a proper initial segment of $M$;

  (iii) For some $\kappa < \nu$ $\beta =_{\tmop{df}} o^N (\kappa) \geq \nu$,
  and if $\zeta < \kappa^{+ M}$ is maximal so that $E^M_{\beta}$ measures all
  subsets of $\kappa = \tmop{crit} (E^M_{\beta})$ which lie in $M\| \zeta$,
  then $N$ is an initial segment of $M^+$ where $\pi : \widehat{M\| \zeta}
  \longrightarrow^{\ast}_{E^M_{\beta}} M^+$.
\end{theorem}

In order to have sufficient further condensation Jensen and Zeman
require certain parameters associated with canonical witness
structures to be in the range of their maps. We only remind the
reader of this definition here, and refer  to the paper for a full
discussion of their significance.

\begin{definition}
  Suppose $\gamma \in p^n_M$ and let $\sigma^M_{\gamma} $be the canonical
  witness map corresponding to $W^{\gamma}_M ;$  if
$  \sup (\tmop{ran} (\sigma^{M}_{\gamma})) \cap \omega \rho^n_{M_{}}
<
  \omega \rho^n_{M}$ we set $\delta (\gamma) = \sup (\tmop{ran}
  (\sigma^{M_{}}_{\gamma}) \cap \omega \rho^n_{M})$.

  We set $\widetilde{p}^n_M =_{\tmop{df}} \{\gamma \in p^n_M \mid \delta
  (\gamma)$ is defined$\}$, and appropriately $\widetilde{p}_M
  =_{\tmop{df}} \bigcup_n \tilde{p}^n_M$.

  Further set $d^n_M =_{\tmop{df}} \{\delta (\gamma)\mid \gamma \in
  \tilde{p}^{}_M, k \leq n\}$ etc.
\end{definition}

This finite (possibly empty) set $d^n_M$ then collects together all those
sups of those canonical witness maps $\sigma_{\gamma}$ just for those $\gamma$
for which the map is non-cofinal at the $k$'th levels for $k \leq n$. This
allows for an appropriate form of the Condensation Lemma for hierarchies below
mice $M$ with any $\kappa$ with $(o_M (\kappa) = \kappa^{+ +})^M$. The
following is again taken from {\cite{JeZe00}}.

\begin{theorem}
  \label{CLII}{\tmem{{\tmstrong{(Condensation Lemma II)}}
  ({\tmem{cf}}{\cite{JeZe00}} 3.1)}} Suppose there is no inner model for $o_M
  (\kappa) = \kappa^{+ +} .$ Let $N, M$ be mice and $\sigma : \hat{N}
  \longrightarrow_{\Sigma_1^{(n)}} \hat{M}$. Suppose further that $\sigma (
  \bar{\alpha}) = \alpha, \sigma ( \bar{p}) = p_M \backslash \alpha,$ and

  (i) $\omega \rho^{n + 1}_M \leq \alpha < \omega \rho^n_M$ \ and $M$ is sound
  above $\alpha ;$

  (ii) $d^n_M \subseteq \tmop{ran} (\sigma)$.

  Then $\bar{p} = p_N \backslash \bar{\alpha}$ ; $N$ is sound above
  $\bar{\alpha}$, $\sigma ( \tilde{p}_N \backslash \bar{\alpha}) = p_M
  \backslash \alpha$ and $\sigma (\delta^N (\gamma)) = \delta^M (\sigma
  (\gamma))$ whenever $\gamma \in \tilde{p}_N \backslash \bar{\alpha}$.
\end{theorem}

We shall need a lemma on preservation of these $d$-parameters under normal
iterations. We prove this here.

\begin{lemma}
  \label{dpreserve}Suppose $\pi : M \longrightarrow N$ is a normal iteration
  of $M$. Then $\pi (d_M) = d_N$.
\end{lemma}

\nod{\bf{Proof: }}This would be by induction on the length of the
iteration, but we simply do a one step ultrapower by an extender $E$
with critical point $\kappa$ and the reader can form the general and
direct limit argument herself. This does not follow quite
immediately from  Condensation Lemma II as the latter assumes $d_N$
is in the range of the map. We know that $\pi (p_M) = p_N$. We may
express

$\overline{p}_M =\{\nu \in p_M \mid $ If $\nu \in [\omega \rho^{k +
1}_M, \omega \rho^k_M)$ then the canonical witness map is\\
\hspace*{9em} non-cofinal into $\omega \rho^k_M \}$.

\nod And: $d_M =\{\delta^M (\nu) | \nu \in \overline{p_{}}_M \}$.

\nod Then if $\delta (\nu) \in d_M$ with $\nu \in [\omega \rho^{k +
1}_M, \omega \rho^k_M)$ we have as in {\tmem{}}{\cite{JeZe00}}:

$(\ast)$ \ \ $\forall \xi^k \forall \zeta^k (\xi^k < \nu \wedge \zeta^k =
\widetilde{h}^{k + 1}_M (\xi^k, p_M \backslash (\nu + 1)) \longrightarrow
\zeta^k \leq \delta (\nu))$.

\nod This is $\Pi^{(k)}_1$ in $\nu, \delta (\nu),$ and $p_M$. If
$\tmop{crit} (E) = \kappa \in [\omega \rho^{n + 1}_M, \omega
\rho^n_M)$ then $\pi$ is $\Sigma_0^{(n)}$ preserving and cofinal
into $\omega \rho^n_M$, hence $\Sigma_1^{(n)}$-preserving. If $k <
n$ then it is $\Sigma_2^{(k)}$preserving. Consequently wherever
$\nu$ lies we have from these preservation properties:

(1) $\nu \in \bar{p}_M \longrightarrow \pi (\nu) \in \bar{p}_N \wedge \pi
(\delta^M (\nu)) \geq \delta^N (\pi (\nu))$ .

{\noindent}We want equality here. For $k < n$ $\Sigma_2^{(k)}$preservation
suffices to guarantee this: if

$\exists \delta^k < \pi (\delta^M (\nu))$[$\forall \xi^k \forall
\zeta^k (\xi^k < \pi (\nu) \wedge \zeta^k = \widetilde{h}^{k + 1}_N
(\xi^k, \pi (p_M)_{} \backslash (\pi (\nu) + 1))$ \\ \hspace*{5em}
$\longrightarrow \zeta^k \leq \delta^k]$

{\noindent}then this would go down to $M$ and give a contradiction.
For $k = n$ we can reason as follows. Suppose $\bar{\delta} = \pi
(f) (\kappa) = \delta^N (\pi (\nu)) < \pi (\delta^M (\nu))$. As at
$(\ast)$:

$\forall \xi^n \forall \zeta^n (\xi^n < \pi (\nu) \wedge \zeta^n =
\widetilde{h}^{n + 1}_N (\xi^n, p_N \backslash (\pi (\nu) + 1))
\longrightarrow \zeta^n \leq \overline{\delta})$.

{\noindent}By using a {\L}o\v{s} Lemma we should have that:

$\{\alpha < \kappa | \forall \xi^n \forall \zeta^n (\xi^n < \nu \wedge \zeta^n
= \widetilde{h}^{n + 1}_M (\xi^n, p_M \backslash (\nu + 1)) \longrightarrow
\zeta^n \leq f (\alpha))\}$

{\noindent}was of $E$-measure 1. But on a set of measure 1 $f
(\alpha) < c_{\delta^M (\nu)} (\alpha) = \delta^M (\nu)$ so this
contradicts the definition of $\delta^M (\nu)$. Hence

(3) $\nu \in \bar{p}_M \longrightarrow \pi (\delta^M (\nu)) = \delta^N (\pi
(\nu))$.

{\noindent}Now note:

(4) $\pi (\nu) \in \bar{p}_N \longrightarrow \nu \in \overline{p}_M$
and hence again $\pi (\delta^M (\nu)) = \delta^N (\pi (\nu))$.

{\noindent}For $k < n$ this follows from $\Sigma_2^{(k)}$
preservation.  For $k = n$ this follows from the cofinality of $\pi$
into $\omega \rho^n_N$: if $\delta^N (\pi (\nu))$ is defined, then
it is less than some $\pi (\delta)$ and the formula $(\ast)$ written
out for
$N$ and $\pi(\delta)$ then goes down to $M$, so this suffices. \mbox{}\hfill  Q.E.D. \\

We shall also be assuming familiarity with the construction of
fine-structural pseudo-ultrapowers, for which see {\cite{Z02}} or
\cite{Wehbk}. We shall be using various ``lift-up'' lemmas. These
are in the following form.

\begin{definition}
  \label{kmn}Let $M$ be an acceptable $J$-structure, and $\nu \in M$ a regular
  cardinal of $M.$ Then $k (M, \nu)$ is defined to be the least $k$ (if it
  exists) so that there is a good {\tmstrong{$\Sigma_1^{(k)}$-}}definable
  function whose domain is a bounded subset of $\nu$ and whose range is
  unbounded in $\nu$. (Such a function is said to {\tmem{singularize}} $\nu$
  and we say that $\nu$ is
{  \boldmath{$\Sigma_1^{(k)}$}}$(M)$-{\tmem{singularized over
$M$}}.)
\end{definition}
\begin{definition}
Let $\bar{M},\nb, k=k(\bar{M},\nb)$ be as above with $\nb$ regular
in $\bar{M}$. Let $\overline{Q}=_{df} J_{\nb}^{\bar{M}} $. Define
 $\Gamma^k_{\bar{M},\nb} =_{\tmop{df}}$
$$\{ f\mid \tmop{dom} ( f )
 \in \overline{Q}\wedge \tmop{ran} ( f ) \subseteq \overline{M} )\,
\wedge (n<k  \wedge\, f
  \in \tmmathbf{\Sigma}^{( n )}_1 ( \overline{M} ) \wedge \omega \rho^{n +
  1}_{\overline{M}} \geq \overline{\nu} )  \}.$$
\end{definition}
\begin{theorem}
  \label{kpseudo}\tmtextbf{{\tmem{(\tmmathbf{}Pseudo-Ultrapower Theorem)}}}
  Let $\bar{M}$ be an acceptable $J$-structure, $\bar{\nu}$ a regular cardinal
  of $\bar{M}$ but with $k = k ( \bar{M}, \bar{\nu})$ defined. Let $\overline{Q}=_{df} J_{\nb}^{\bar{M}} $.
  Then there is a
  map $\tilde{\sigma} : \bar{M} \longrightarrow_{\Sigma_0} M$ (the ``canonical
  $k$-extension'' of $\sigma : \overline{Q} \longrightarrow_{\Sigma_0} Q$)
  satisfying:

  (i) $\tilde{\sigma}$ is $Q$-preserving, $M$ is an acceptable end extension
  of $Q$, and

  $M =\{ \tilde{\sigma} (f) (u) \hspace{0.25em} | \hspace{0.25em} u \in
  \sigma (\tmop{dom} (f)), f \in \Gamma\}$.

  (ii)a) $\tilde{\sigma}$ is $\Sigma^{(n)}_2$ preserving for $n < k$;\\
  \ \ \ \ \ \ b) $\rho_k = \rho^k_M$, and $\tilde{\sigma}$ is
  $\Sigma^{(k)}_0$ preserving and cofinal (thus
  $\Sigma_1^{(k)}$-preserving);\\
  (iii) $\tilde{\sigma} ( \bar{\nu}) = \nu$ and the latter is regular in $M$
  ;\\
  (iiv)  $k = k (M, \nu)$: $k$ is least so that there is a $\tmmathbf{\Sigma_1^{(k)}} (M)$ map
  cofinalising $\nu$.
\end{theorem}

\begin{lemma}
  \label{interpolation}{\tmem{{\tmstrong{(Interpolation Lemma)}}}} Suppose
  $\overline{M} = \langle J^{\overline{A}}_{\overline{\beta}}, \overline{B}
  \rangle$ is a structure such that $\overline{\nu}$ is regular in
  $\overline{M}$, but with $k = k ( \bar{M}, \bar{\nu})$ defined. Suppose
  further that $f : \overline{M} \longrightarrow_{\Sigma_1^{(k)}} M = \langle
  J^A_{\beta}, B \rangle$. Let $\widetilde{\nu} = \sup f$``$\bar{\nu}$. Then
  there is a structure $\widetilde{M} = \langle
  J^{\widetilde{A}}_{\widetilde{\beta}}, \widetilde{B} \rangle$, a map
  $\widetilde{f} : \overline{M} \longrightarrow \widetilde{M}$ with $
  \widetilde{f} \supseteq f \upharpoonright J^{\overline{A}}_{\overline{\nu}}$
  and $\widetilde{f}, \Sigma_0^{(k)}$-cofinal (and hence
  $\Sigma_1^{(k)}$-preserving), and a unique $f' : \widetilde{M}
  \longrightarrow_{\Sigma_0^{(k)}} M$, with $f = f' \circ \widetilde{f} $ and
  $f' \upharpoonright \tilde{\nu} = \tmop{id} \upharpoonright \tilde{\nu}$.
\end{lemma}

\section{Global $\LBox$ in $K$.}

\begin{definition}
  \label{globalsquare}Let $\tmtextrm{\tmop{Sing}} =\{\beta \in
  \mathrm{\tmop{Ord}} \mid \lim (\beta) \wedge \mathrm{\tmop{cf}} (\beta) <
  \beta\}$ be the class of singular limit ordinals. {\em Global} $\square$
   is the assertion: there is a system $(C_{\beta})_{\beta \in
  \mathrm{\tmop{Sing}}}$ satisfying:\\
  {\hspace{2em}}(a) $C_{\beta}$ is a closed cofinal subset of $\beta$;\\
  {\hspace{2em}}(b) $\mathrm{\tmop{ot}} (C_{\beta}) < \beta ;$\\
  {\hspace{2em}}(c)\label{Coherence} if $\overline{\beta}$ is a limit point of
  $C_{\beta}$ then $\overline{\beta} \in \mathrm{\tmop{Sing}}$ and
  $C_{\overline{\beta}} = C_{\beta} \cap \overline{\beta}$.
\end{definition}

Jensen {\cite{Je72}} introduced the principle  and proved it held in
$L$. The format of the proof we shall follow will be that of
{\cite{BJW}}, which  was a proof in the setting of generalised $L
[A]$ hierarchies suitable for use Jensen's Coding Theorem. {The
second author {\cite{Welch}} proved  in the Dodd-Jensen core model
$K$.} The first proof of $\LBox$ which used the Baldwin-Mitchell
arrangement of the $L [E]$ hierarchy, was for Jensen's model for $K$
with measures of order zero, and was by \tmtextsc{Wylie}
{\cite{Wy90}}. From the order types of the square sequences
$C_{\xi}$ we shall define stationary sets $S_n$ to which we shall
apply the $\tmop{MS}$-principle.

We consider how a global $\LBox$ sequence can be derived in $K$. For
clarity we shall assume there is no inner model with a measure of
Mitchell order $o_M (\kappa) = \kappa^{+ +}$ (see {\cite{JeZe00}})
and that $K$ is built under this assumption. We assume for the rest
of this section $V = K.$ Jensen and Zeman prove (more than) the
following.

\begin{theorem}
  \label{ksquare} Let $_{} S$ be the class of all singular limit ordinals that
  are limits of admissibles. There is a uniformly definable class $\langle
  C_{\nu} | \nu \in S \rangle$ so that:

  (i) $C_{\nu}$ is a set of ordinals closed below $\nu$ and, if $\tmop{cf}
  (\nu) > \omega$, then it is also unbounded;

  (ii) $\tmop{ot} (C_{\nu}) < \nu$;

  (iii) $\overline{\nu} \in C_{\nu^{}}^{} \longrightarrow \overline{\nu} \in S
  \wedge C_{\bar{\nu}} = \overline{\nu} \cap C_{\nu} ;$
\end{theorem}

It is well known that once one has a global  sequence defined on the
singular ordinals of some cub class that contains all singular
cardinals and is cub beneath each successor cardinal, then this can
be filled out to a global sequence on {\tmem{all}} singular ordinals
to satisfy Definition \ref{globalsquare}. Hence proving the above
theorem suffices. As $V = K = L [E]$ for $E = E^K $ a fixed sequence
of extenders, if $\nu$ is a singular ordinal, then there will be a
least level $J^E_{\beta (\nu)}$ of the $J^E$-hierarchy over which
$\nu$ is definably singularised, {\tmem{i.e.}} there will be a
partial {\bm$\Sigma_{\omega}$}$( J^E_{\beta (\nu)}$) definable good
function  mapping a subset of some $\gamma$ cofinally into $\nu$.
This level of the hierarchy $J^E_{\beta (\nu)} $ will also be our
main singularising structure $M_{\nu}$. Note that by Lemma
\ref{Levy=*} and the soundness of the $K$ hierarchy, any such
function is also $\Sigma_1^{(n)}$($J^E_{\beta (\nu)})$ for some $n$.
That is, $k (\nu, J^E_{\beta (\nu)})$ in the sense of Definition
\ref{kmn} is defined.

However there will be many other mice over which ordinals are singularized
and we must consider these in addition.

\begin{definition}
  $S^+$ is the class of $s = \langle \nu_s, M_s \rangle$ where

  (a) $\nu_s \in \tmop{Sing} ;$

  (b) $M_s$ is a mouse satisfying the following:

  (i) $\nu_s$ is regular in $M_s$ and $J_s =_{\tmop{df}} J_{\nu_s}^{E^{M_s}}$
  is a union of admissible sets $J_{\tau}^{E^{M_s}}$;

  (ii) for some $m$, $\nu_s$ is $\tmmathbf{\Sigma^{(m)}_1} (M_s)$
  singularised, that is $k (\nu_s, M_s)$ is defined;

  (iii) $M_s $is sound above $\nu_s$, and if $\nu_s = \kappa^{+ M_s}$ where
  $\kappa \in \tmop{Card}^{M_s},$ then $M_s$ is sound above $\kappa$.
\end{definition}

Recall that if $M=\langle J_\alpha^E,\in\rangle$ and $\nu\leq
\alpha$ then $M| | \nu =_{df} \langle J_\nu^E,\in\, E_\nu \rangle$.
 We then note the following facts:

\begin{lemma}
  (i) \ If $\langle \nu, M \rangle, \langle \nu, N \rangle$ satisfy
  (b)(i),(ii) above but are both sound above $\nu$, with $M| | \nu = N| |
  \nu$,  then $M = N$.

  (ii) If $\langle \nu, M \rangle, \langle \nu, N \rangle \in S^+ $and
  $J_{\nu}^{E^M} = J_{\nu}^{E^N}$ then $M = N$.
\end{lemma}

{\noindent}{\bf{Proof: }} Straightforward iteration and comparison.
\mbox{}\hfill  Q.E.D. \\

The following definition encapsulates the essential concepts
associated with  singularising structures.

\begin{definition}
  \label{squareobjects} Let $s \in S^+_{} .$ Then we associate the following
  to $\nu_s$:
  \begin{enumeratealpha}
    \item $n_s =_{\tmop{df}}  k (\nu_s, M_s)$, the least $n \in \omega$ so
    that $\nu_s$ is $\tmmathbf{\Sigma^{(n)}_1} (M_s)$ singularised over $M_s
    .$

    \item $M^l_s =_{\tmop{df}} M^{\,l,\, p_{{M_s}} \upharpoonright l}_s$ for $l
    \leq k_s, n_s .$

    \item $h^l_s =_{\tmop{df}} h^{\,l,\, p_{{M_s}} \upharpoonright l}_{M_s}$;
    \ \ \ $h_s =_{\tmop{df}} h^{n_s}_s ;$ \ \ \ $\widetilde{h_{}}_s
    =_{\tmop{df}} \widetilde{h}_{M_s}^{n_s + 1} .$

    \item $\kappa_s \simeq$ the largest cardinal of $J_s$, if such exists;
    $\omega \rho_s =_{\tmop{df}} \tmop{On} \cap M^{n_s}_s$;\\ $\beta (s)
    =_{\tmop{df}} \tmop{On} \cap M_s .$

    \item $p_s =_{\tmop{df}} p_{M_s} \backslash \nu_s$ if $\nu_s$ is a limit
    cardinal of $J_s ;$ $p_s =_{\tmop{df}} p_{M_s} \backslash \kappa_s$
    otherwise; \\ $q_s =_{\tmop{df}} p_s \cap \omega
    \rho^{n_s}_{M_s} ;$\\ $d_s =_{\tmop{df}} d^{}_{M_s}$

    \item $\alpha_s =_{\tmop{df}} \max \{\alpha < \nu | \nu \cap \tilde{h}_s
    (\alpha \cup \{p_s \}) = \alpha\}$, setting $\max {\O}= 0.$

    \item $\gamma_s \simeq \min \{\gamma < \nu | \exists f (f$ a good
    ${{\Sigma^{(n_s)}_1}}\mbox{} ^{M_s} (\{p_s \})$ function singularising $\nu$ with
    domain $\subseteq \gamma)\}$.
  \end{enumeratealpha}
\end{definition}

Thus if $\nu_s = \kappa_s^+$, we may have $\kappa_s$ in $p_s .$ Note
that the closure of the set in {\tmem{f}}) ensures that $\alpha_s$
is always defined; note also that $\alpha_s$ must be strictly less
than the first ordinal $\gamma_s$ partially mapped by
$\widetilde{h_{}}_s$ (with parameter $p_s)$ cofinally into $\nu_s$.
 Note also that if we  set $\gamma' = \max \{\gamma_s, (p_{M_s}
\cap \nu_s$)+1\} \ ($\max \{\gamma_s, (p_{M_s} \cap \kappa_s$)+1\}
if $\kappa_s$ is defined), and then $\widetilde{h}_s (\gamma' \cup
p_s)$ must be cofinal in $\nu_s$ since we shall have enough
parameters in the domain of this hull to define our cofinalising
map).

\begin{lemma}
  \ \ $\omega \rho^{n_s}_{M_s} \geq \nu \geq \omega \rho^{n_s + 1}_{M_s}$
\end{lemma}

\nod{\tmstrong{{\noindent}Proof}} Let $n = n_s$, $\nu = \nu_s$.
Suppose the first inequality failed. Then $n > 0$, and we have some
parameter $q$ with a $\Sigma_1^{(n-1)} (M_s)(\{q\})$ partial map $f$
of some $\gamma < \nu$ cofinal into $\nu$.

Pick such a $\gamma > \omega \rho^{n_{}}_{M_s}$. Let $\pi :
\widetilde{M} \longrightarrow M_s$ have range $\widetilde{h^{}}_s^n
(\gamma \cup \{q, p_s \})$, with $\widetilde{M}$ transitive. By the
leastness of $n$, $\tmop{ran} (\pi)$ cannot be unbounded in $\nu$. \
By Lemma \ref{sigma*}, since $\pi \upharpoonright \omega
\rho^{n_{}}_{M_s} = \tmop{id}$ and $p_{M_s} \in \tmop{ran} (\pi),$ \
$\pi$ is $\Sigma^{\ast}$ elementary. However then $\tmop{ran} (f)
\subseteq \tmop{ran} (\pi)$, with the former unbounded in $\nu$. A
contradiction! If the second inequality failed, then the partial
function $\tmmathbf{\Sigma}_1^{(n)} (M_s)$ singularising $\nu$ would
be a subset of $\nu_{}$ and thus a bounded subset of $\omega \rho^{n
+ 1}_{M_s}$ belonging to $M_s$. \mbox{}\hfill  Q.E.D.

\begin{definition}
  \label{basicf}For $s, \bar{s} \in S^+ :$(i) We set $f : \bar{s}
  \Longrightarrow s$ if there is $|f|$ with $|f| : J_{\overline{s}}
  \longrightarrow_{\Sigma_1} J_s,$ and $|f|$ is the restriction of some
  $f^{\ast} : M_{\bar{s}} \longrightarrow_{\Sigma^{(n)}_1} M_s$ where $n =
  n_s, \nu_s = f^{\ast} (\nu_{\bar{s}})$(if $\nu_s \in M_s) ; \kappa_s \in
  \tmop{ran} (|f|)$ (if $\kappa_s$ is defined); $\alpha_s, p_s, d_s$ are all
  in $\tmop{ran} (f^{\ast}) .$

  (ii) $\mathbb{F}=\{ \langle \bar{s}, |f|, s \rangle | f : \bar{s}
  \Longrightarrow s \}$; we write here $\bar{s} = d (f), s = r (f) ;$

  (iii) If $\nu_s \in M_s$, we set:

   $\quad p (s) =_{\tmop{df}} p_s \cup \{d_{s,}
  \alpha_s, \nu_s, \kappa_s \}$ (if $\kappa_s$ is defined);
  otherwise

  $\quad p (s) =_{\tmop{df}} p_s \cup \{d_s, \alpha_s, \kappa_s \}$ (again
  including $\kappa_s$ only if it is defined).

  (iv) $f_{(\delta, q, s)}$ is the inverse of the transitive collapse of the
  hull $\widetilde{h_{}}_s (\delta, \{p (\nu)\})$ in $M_s$.
\end{definition}

  \label{remarkf}(Lemma \ref{4.3}  will justify  in the final clause {\tmem{(iv)}} that
  there
  is some $\bar{s}$ so that $\langle \bar{s}, |f_{(\delta, q, s)} |, s \rangle
  \in \mathbb{F}$).

\begin{lemma}
  \label{determined}If $\exists \bar{s} (f : \bar{s} \Longrightarrow s$) then
  $|f|$ and $f^{\ast}$ are uniquely determined by $\tmop{ran} (|f|) \cap_{}
  \nu_s$.
\end{lemma}

\nod{\bf{Proof:}} As $M_s$ is sound above $\nu_s$, we have by our
definitions, that $\widetilde{h}_s (\omega \nu_s \cup \{p_s \}) =
M_s$. We have a $\Delta_1 (J_s)$ onto map $g : \omega \nu_s
\twoheadrightarrow J_s$. Thus, if $Y = \widetilde{h}_s (\omega \nu_s
\cap \tmop{ran} (|f|) \cup \{p_s \}) \text{, then $Y =
\widetilde{h}_s (\tmop{ran} (|f|) \cup \{p_s \})$=$\tmop{ran}
(f^{\ast})$}$.\hfill Q.E.D.\\


Lemma \ref{determined} justifies us in calling  $f^{\ast}$ the
{\tmem{canonical extension}} of $f$, (or rather $|f|$) and sometimes
we abuse notation and write $f^{\ast} : J_{\overline{s}}
\longrightarrow_{\Sigma_1} J_s$ where more correctly we should write
$f^{\ast} \upharpoonright J_{\overline{s}} : J_{\overline{s}}
\longrightarrow_{\Sigma_1} J_s$.
By virtue of the last
lemma, this does not cause any ambiguity.

The next two lemmata are fundamental and concern relationships
between singularising structures, and associated maps between them.

\begin{lemma}
  \label{4.3}Let $f : \overline{M} \longrightarrow_{\Sigma_1^{(n)}} M_s$;
  suppose $f ( \bar{d}, \bar{\alpha}, \bar{p}) = d_s, \alpha_s, p_s$, and
  (where appropriate) $f ( \bar{\kappa}, \bar{\nu}) = \kappa_s, \nu_s$. (The
  latter if $\nu_s \in M_s$ ; if $\nu_s = \tmop{On} \cap M_s$ then we take
  $\bar{s} = \tmop{On} \cap M_{\bar{s}} .)$ Then \ $\bar{s} = ( \bar{\nu},
  \overline{M}) \in S^+_, $ and thus $f : \bar{s} \Longrightarrow s$; moreover
  $n, \bar{d}, \bar{\alpha}, \bar{p}, \bar{\kappa}$ (the latter defined if
  $\kappa_s$ is) \ are $n_{\bar{s}}, d_{\bar{s}}, \alpha_{\bar{s}},
  p_{\bar{s}}, \kappa_{\bar{s}}$.
\end{lemma}


{\tmstrong{Proof}} \ We shall show that $\overline{M}$ is a
singularising structure for $\bar{\nu} =_{\tmop{df}} \nu_{\bar{s}}$
and the other mentioned parameters have the requisite properties to
satisfy the relevant definitions, and  are moved correctly by $f$.
We set $\nu = \nu_s$.

(1) $\bar{p} = p_{\bar{M}} \backslash \overline{\nu} ; \bar{d} = d^n_{\bar{M}}
\backslash \bar{\nu}${\tmem{ and}} $\overline{M}${\tmem{ is sound above
$\bar{\nu}$.}}

Proof: Directly by the Condensation Lemma II Theorem \ref{CLII}
\mbox{}\hfill Q.E.D. (1) \\

Let $\overline{h}$ have the same functionally absolute definition
over \ $\overline{M}$ as $\widetilde{h}_s $ does over $M_s.$
$\overline{h}$ is thus \ $\tmmathbf{\tmmathbf{\Sigma}_1^{(n)}} (
\overline{M}) .$

(2) {\tmem{$\overline{\alpha} $ is defined from $ \overline{M}_{}$
as $\alpha$ was defined from }}$M_s .$

Proof: Set $H (\xi^n, \zeta^n) \longleftrightarrow \widetilde{h}_s (\omega
\xi^n \cup \{p_s \}) \cap \nu \subseteq \zeta^n$

\quad\quad $\overline{H} (\xi^n, \zeta^n) \longleftrightarrow
\bar{h} (\omega \xi^n \cup \{ \bar{p} \}) \cap \bar{\nu} \subseteq
\zeta^n .$

Then $H$ is $\Pi_1^{_{} (n) M_s^{}} (\{p_s \}),$  and $\overline{H}$
$\Pi^{(n) \bar{M}}_1 (\{ \bar{p} \}),$ by the same definition. As
$M^{}_s \models H (\alpha, \alpha)$ it follows that $\overline{M}
\models \overline{H} ( \bar{\alpha}, \bar{\alpha})$. However for any
$\xi$ with $\bar{\alpha} < \xi^n < \bar{\nu}$ we must have
$\overline{M} \models \neg \overline{H} (\xi^n, \xi^n)$, because
$M^{}_s \models \neg H (f (\xi^n), f (\xi^n))$ since $\alpha < f
(\xi^n) < \nu$. Hence $\bar{\alpha}$ is defined in the requisite
way. \mbox{}\hfill Q.E.D. (2)\\

(3) $\exists \overline{\xi}^n < \bar{\nu} ( \tilde{h}_s (f ( \bar{\xi}^n) \cup
\{p_s \})$ \tmtextit{is unbounded in} $\nu)$.\\

{\noindent} If (3) were to hold for some $\overline{\xi}$ then
$\bar{h} (\xi \cup \{ \bar{p} \})$ would be cofinal in $\bar{\nu}$,
since the following is a $\Pi^{(n)}_2$ expression which thus would
go down to $\overline{M}$.  It would then be a statement about the
parameters $\bar{p}, \bar{h}$, $\bar{\xi}$ and $\bar{\nu}$ (the
latter if $\nu = f ( \bar{\nu}) < \omega \rho_s$):

$M_s \models (\forall \zeta^n < \nu) (\exists \delta^n < f (
\bar{\xi}^n)) (\exists i < \omega) (\zeta^n < \tilde{h}_s (i,
\langle \delta^n, p_s \rangle) < \nu).$

{\noindent} This would show that $\bar{h}$ is a  singularising
function for $\bar{\nu}$ over the structure $\overline{M}$ and that
$n \geq n_{\bar{s}}$. We need to show that (3) holds. Suppose not.
This has the consequence that $\tau =_{\tmop{df}} \sup
f$``$\bar{\nu} \leq \gamma_s < \nu$. As $\alpha_s\in\tmop{ran} (f
\upharpoonright \bar{\nu})$ we have that $\alpha_s <\tau$. So by
definition of $\alpha_s$ itself:\\

(4) $\tau \neq \nu \cap \tilde{h}_s (\tau \cup \{p_s \})$.\\

 Hence the following
is true in $M_s$:
\[ \exists i \in \omega \exists \xi^n < \tau (\nu > \tilde{h}_s (i, \langle
   \xi^n, p_s \rangle) \geq \tau) . \]
{\noindent} Let $i, \xi^n$ witness this, and pick $\bar{\delta} < \bar{\nu}$
so that $f ( \bar{\delta}) > \xi^n .$ Then for any $\bar{\mu} < \bar{\nu},$ as
$f ( \bar{\mu}) < \tau$:

$M_s \models (\exists \zeta^n < f ( \bar{\delta})) (\nu > \tilde{h}_s (i,
\langle \zeta^n, p_s \rangle) \geq f ( \bar{\mu}))$.

{\noindent} This is $\Sigma^{(n)}_1$ and hence, for all $\bar{\mu} <
\bar{\nu}$, goes down to $\overline{M}$, yielding:

$\overline{M} \models \forall \bar{\mu} < \bar{\nu} \exists \zeta^n <
\bar{\delta} ( \bar{\nu} > \bar{h} (i, \langle \zeta^n, \bar{p} \rangle \geq
\bar{\mu})$.

Hence $\bar{h}$ is a  singularising function for $\bar{\nu}$. Thus
whether (3) holds or not we have established the existence of
suitable $\Sigma_1^{(n)} ( \overline{M})$ singularising function.

(5) $n = n_{\bar{s}}$.

We are left with showing $n \leq n_{\bar{s}}$, as the above shows
that $n \geq n_{\bar{s}}$. Suppose $m < n$ and that $\bar{g}$ is a
$\Sigma_1^{(m)} ( \overline{M})$ good function in the parameter
$\bar{r}$. Let $g$ be $\Sigma_1^{(m)} (M_s)$ using the same
functionally absolute definition and the parameter $f ( \bar{r})$.
Suppose $\bar{\delta} < \bar{\nu}$. By the
$\Sigma^{(n)}_1$-elementarity of $f$ we have the following
$\Sigma_1^{(n)}$ statement holds in $M_s$ (as ran($g \upharpoonright
f ( \bar{\delta}))$ is bounded in $\nu$):
\[ (\exists \xi^n < \nu) (\forall \zeta^m < f ( \bar{\delta})) (\forall \eta^m
   < \nu) (g (\zeta^m) = \eta^m \longrightarrow \eta^m < \xi^n) \]
{\noindent} (assuming $\nu < \beta ;$ otherwise drop the bound $\nu
.)$ As $f$ is  $\Sigma_1^{(n)}$-preserving, we have in
$\overline{M}$:
\[ (\exists \xi^n < \bar{\nu}) (\forall \zeta^m < \bar{\delta}) (\forall
   \eta^m < \bar{\nu}) ( \bar{g} (\zeta^m) = \eta^m \longrightarrow \eta^m <
   \xi^n) \]
{\noindent} As $\bar{\delta}$ was arbitrary, we conclude $\tmop{ran}
( \bar{g} \upharpoonright \xi)$ is bounded on any $\xi < \bar{\nu}
.$ Hence $n \leq n_{\bar{s}} .$  \mbox{ }\hfill Q.E.D.(5) and Lemma.\\

\begin{definition}
  Suppose $f : \bar{s} \Longrightarrow s.$ Then let $\lambda (f) =_{\tmop{df}}
  \sup f$``$\overline{\nu}$; $\rho (f) =_{\tmop{df}} \sup
  f$``$\rho_{\bar{\nu}} .$
\end{definition}

\begin{lemma}
  \label{initseg}Suppose $f : \bar{s} \Longrightarrow s,$ and let $\lambda =
  \lambda (f)$. Then $\lambda \in \tmop{Sing}$ and there exists a unique $f_0
  : \bar{s} \Longrightarrow s' = s| \lambda$ with $f \upharpoonright
  \overline{\nu} = f_0 \upharpoonright \overline{\nu} .$
\end{lemma}

\nod{\bf Proof: }\ Let $n = n_s$. We apply directly the
Interpolation Lemma with $\lambda$ as $\widetilde{\nu}$,
$M_{\bar{s}}$, $M_s$ as $\overline{M}, M$ respectively, and using
$f^{\ast} : M_{\overline{s}} \longrightarrow_{\Sigma_1^{(n)}} M_s$
(where $f^{\ast}$ is the canonical extension of $f)$ we have the
structure $\widetilde{M} = M_{s'}$ and maps $\widetilde{f}, f'$ as
specified. \ \


(1) $ s' = \langle \lambda, \widetilde{M} \rangle \in S^+, n = n_{s'} .$

By the comment above $\gamma_{\bar{s}}$ is defined and $n =
n_{\bar{s}}$. As $\widetilde{h}_{\bar{s}} (\gamma_{\bar{s}} \cup
\{p_{M_{\bar{s}}}, r\})$ is cofinal in $\overline{\nu}$ for some
parameter $r$ then $\lambda \cap \widetilde{h^{}}^{n +
1}_{\widetilde{M}} ( \widetilde{f} (\gamma_{\! \bar{s}}) \cup \{p',
\widetilde{f} (r)\})$ is cofinal in $\lambda$ (setting $p' =
\widetilde{f} (p_{\! \overline{s}}) = f'^{- 1} (p_s)$). Thus
$\lambda$ is $\Sigma_1^{(n)}$-singularised over $\tilde{M}$. Hence
$n \geq n_{s'} .$ \ We need to show that $\lambda$ is not
$\Sigma_1^{(n - 1)}$-singularised over $\tilde{M}$. Suppose this
fails and thus that $\{\alpha | \sup (\lambda \cap
\widetilde{h^{}}^n_{\widetilde{M}} (\alpha \cup \{r\})) = \alpha\}$
is bounded in $\lambda$, by $\alpha'$ say, for some choice of a
parameter $r \in \widetilde{M} = M_{s'}$. By the construction of the
pseudo-ultrapower we may assume that $r$ is of the form
$\widetilde{f} ( \bar{g}_0) (\eta)$ for some good
$\tmmathbf{\Sigma}^{(n - 1)}_1 (M_{\overline{s}})$ function
$\overline{g}_0$ and some $\eta < \lambda$. Define\\

\nod$\widetilde{H} (\xi^n, \zeta^n, d) \longleftrightarrow
\widetilde{h^{}}^n_{\tilde{M}} (\omega \xi^n \cup \{d\}) \cap
\lambda \subseteq \zeta^n ; \bar{H} (\xi^n, \zeta^n, d)
\longleftrightarrow \widetilde{h}^n_{\bar{s}} (\omega \xi^n \cup
\{d\}) \cap \bar{\nu} \subseteq \zeta^n .$\\

These are (uniformly defined) $\Pi^{(n)}_1$ relations over their
respective structures - in the parameters $\lambda, \bar{\nu}$. By
the leastness in the definition of $n_{\bar{s}} $ we have that there
are arbitrarily large $\bar{\tau}^n < \bar{\nu}$ with
$\widetilde{h}^n_{\bar{s}} (\omega \bar{\tau}^n \cup \{p_{\bar{s}}
\}) \cap \bar{\nu} \subseteq \bar{\tau}^n$; using the soundness of
$M_{\bar{s}}$ above $\bar{\nu}$, this implies that for arbitrary
$\zeta^n < \bar{\tau}$: $ \widetilde{h}^n_{\bar{s}} (i,
\overline{\xi^{}}^n, \bar{g}_0 (\zeta^n)) \cap \bar{\nu} \subseteq
\bar{\tau}^n$. In other words:

$\forall \zeta^n < \bar{\tau}^n  \overline{H} ( \bar{\tau}^n, \bar{\tau}^n,
\bar{g}_0 (\zeta^n))$.

As the substituted $\bar{g}_0$ is good $\Sigma_1^{(n - 1)}$ we have that this
is a $\Pi^{(n)}_1$ statement, and so is preserved upwards to $M_{s'} :$

$\forall \zeta^n < \widetilde{f} ( \bar{\tau}^n) \tilde{H} ( \widetilde{f} (
\bar{\tau}^n), \widetilde{f} ( \bar{\tau}^n), \widetilde{f} ( \bar{g}_0)
(\zeta^n))$.

However as $\widetilde{f} \upharpoonright \bar{\nu}$ is cofinal into
$\lambda$, we may choose $\bar{\tau}^n$ so that $\widetilde{f} ( \bar{\tau}^n)
> \max \{\alpha', \eta\}$. This contradicts our definition of $\alpha'$. \ \ \
\ \mbox{}\hfill  Q.E.D.(1)\\

(2) $p' = p_{s'} .$

By the pseudo-ultrapower construction, we have $\widetilde{M} =
\widetilde{h^{}}^{n + 1}_{\widetilde{M}} (\lambda \cup p')
=\widetilde{h^{}}^{n + 1}_{\widetilde{M}} ( \tilde{\kappa} \cup p')$
(where $\tilde{\kappa} = \widetilde{f} (\kappa_{\overline{s}})$ if
$\kappa_{\overline{s}}$ is defined) and is sound above $\lambda$ (or
$\tilde\kappa$). The solidity of $p_{\overline{s}}$ above $\bar\nu$
transfers via the $\Sigma^{(n)}_1$-preserving map $f'$ to show that
$p'$ is solid above $\lambda$ (see \cite{Z02} 3.6.8). Then the
minimality of the standard parameter and the definition of $p_{s'}$
shows that $p_{s'} \leq^{\ast} p'$. However if $p_{s'} <^{\ast} p'$
held, we should have for some $i \in \omega, \vec{\xi}$ that $p' =
\widetilde{h^{}}^{n + 1}_{\widetilde{M}} (i, \langle \vec{\xi},
p_{s'} \rangle)$, and thus $p_s = \widetilde{h^{}}^{n + 1}_s (i,
\langle f' ( \vec{\xi}), f' (p_{s'}) \rangle)$ whence $M_s =
\widetilde{h^{}}^{n + 1}_s (\nu \cup f' (p_{s'}))$. This is a
contradiction as $ f' (p_{s'}) <^{\ast} p_s .$ \mbox{}\hfill
Q.E.D.(2)\\

(3) If $\widetilde{d} =_{\tmop{df}}  \widetilde{f} (d_{\overline{s}}$) then
$\widetilde{d} = d_{s'} .$

Proof: This is very similar to Lemma \ref{dpreserve}, using the
$\Sigma_1^{(n)}$-preservation properties of $\widetilde{f}$, and is
left to the reader. \ \ \mbox{}\hfill Q.E.D.(3)\\

(4) If $\widetilde{\alpha} =_{\tmop{df}}  \widetilde{f}
(\alpha_{\overline{s}}$) then $\widetilde{\alpha} = \alpha_{s'} .$

That $\widetilde{\alpha}$ is sufficiently closed, and hence
$\widetilde\alpha\leq \alpha_{s'}$, is proven as in (2) of Lemma
\ref{4.3} using: \ $\widetilde{H} (\xi^n, \zeta^n)
\longleftrightarrow h_{s'} (\omega \xi^n \cup \{p_{\lambda} \}) \cap
\lambda \subseteq \zeta^n$; $\bar{H} (\xi^n, \zeta^n)
\longleftrightarrow h_{\bar{s}} (\omega \xi^n \cup \{p_{\bar{s}} \})
\cap \bar{\nu} \subseteq \zeta^n .$ For $\widetilde{\alpha}
<\eta^n<\lambda$ we  set $\bar\eta = f^{ -1}$``$\eta^n$. Then we
have $\neg \widetilde{H}(\bar\eta,\bar\eta)$ (as $\bar\eta
>\alpha_{\bar s}$). Hence for some $i\in \omega$, some
$\bar\xi<\bar\eta$ we have $\eta\leq h_{\bar s}(i,\langle
\bar\xi,p_{\bar s}\rangle) < \bar \nu$. As $f(\bar\eta)\geq \eta^n$
and as $\widetilde{f}$ is $\Sigma^{(n)}_0$-preserving we have
$\eta^n\leq h_{s'}(i,\langle \widetilde{f}(\bar\xi),p_{ s'}\rangle)
< \lambda$. \mbox{}\hfill Q.E.D.(4)

We have shown enough now to set that $f^{\ast}_0 = \widetilde{f}$.
\mbox{}\hfill Q.E.D.(Lemma)
\begin{lemma}
  \label{suprema}Suppose $f : \bar{s} \Longrightarrow s$ and $k_s = n_s$. Then
  $\lambda (f) < \nu_s \longleftrightarrow \rho (f) < \rho_s$.
\end{lemma}

\nod{\bf Proof: }$(\rightarrow)$ Suppose $\rho (f) = \rho_s$. \ Let
$\lambda = \lambda (f) .$ Then, in the notation of the previous
Lemma the map $f'$ is not only $\Sigma_0^{(n)}$ but is cofinal at
the $n$'th level, and thus $\Sigma_1^{(n)}$-preserving. We also have
that $f' (\langle\lambda, p_{s'}\rangle) = \langle\nu, p_s\rangle$.
This implies that $\nu \cap f'$``$ h_{s'} (\lambda \cup p_{s'})
\subseteq \nu \cap h_s (\lambda \cup p_s) = \lambda .$ Were $\lambda
< \nu$ this would contradict the fact that $\lambda > \alpha_s$ as
the latter is by supposition, in $\tmop{ran} (f)$.

$(\leftarrow)$ Suppose $\lambda =_{\tmop{df}} \lambda (f) = \nu$.
Again in the same notation, suppose $\rho' =_{\tmop{df}} \rho (f) <
\rho_s,$ It is then easy to see that a good $\Sigma^{(n)}_1$
function, $\overline{F}$ say, singularizing $\bar\nu$ definable in
some parameter ${\overline{q}}$ is taken by the
$\Sigma^{(n)}_0$-preserving $f^\ast$ to a good $\Sigma^{(n)}_1(M_s)
$ function $F$ in $q=f({\overline{q}})$  singularizing $\lambda$,
with all the parameters of the form $x^n$ needed to define the
values $F(\xi)$ in $ran(f^\ast)$. However if $\rho' =
< \rho_s$  we should have that $F\in M_s$. However $\lambda = \nu$!
Contradiction!
\hfill Q.E.D. \\

The construction of the $C_s$-sequences attached to $s = (\nu_s, M_s)$ will
follow in essence the construction in {\cite{Wehbk}}. The main point is that
we can give an estimate to the length of the $C_s$ sequence.

We may state immediately what the $C_s$-sequences for $s = (\nu_s,
M_s) \in S^+ $ will be:

\begin{definition}
  Let $s \in S^+$; $C_s^+ =_{\tmop{df}} \{\lambda (f) \mid s\}; C_s
  =_{\tmop{df}} C_s^+ \backslash \{\nu_s \}$.
\end{definition}

\begin{definition}
  Let $f : \bar{s} \Longrightarrow s.$ Then $ \beta (f) =_{\tmop{df}}
  \max \{\beta \leq \alpha_{\nu_s} \mid f \upharpoonright \beta = \tmop{id}
  \upharpoonright \beta\}$.
\end{definition}

By elementary closure considerations show that $\beta (f)$ is
defined, and  that $\beta (f) = \alpha_{\nu_s}$ iff $f =
\tmop{id}_{\nu_s}$ iff $f (\beta) \ngtr \beta$. if $\beta (f)$ were
singular in $M_{\bar{\nu}}$ using some cofinal function
$g:\beta'\longrightarrow\beta$ with $\beta'<\beta$, we should have
that then $\beta (f)
> \sup(ran(g))= \beta$. Hence $M_{\bar{\nu}} \models$``$\beta (f)$
is a regular cardinal''.

The next lemma lists some properties of $f_{(\gamma, q, s)}$ which
were defined at \ref{basicf}. Firstly a {\tmem{minimality property}}
of $f_{(\gamma, q, s)}$.

\begin{lemma}
  \label{minimality}(i) If $\gamma \leq \nu_s$ then $f_{(\gamma, q, s)}$ is
  the least $f$ such that $f \upharpoonright \gamma = \tmop{id}
  \upharpoonright \gamma$ with $q, p (s) \in \tmop{ran} (f^{\ast})$, in that
  if g is any other such with these two properties, (meaning that $g
  \Longrightarrow s$ with extension $g^{\ast}$ so that $\gamma \cup \{q, p
  (s)\} \subseteq \tmop{ran} (g^{\ast})$) \ then $g^{- 1} f_{(\gamma, q,
  \nu_s)} \in \mathbb{F}.$

  (ii) $f_{(\gamma, q, s)} = f_{(\beta, q, s)}$ where $\beta = \beta
  (f_{(\gamma, q, s)})$.

  (iii) $f_{(\nu, 0, s)} = \tmop{id}_s $;

  (iv) Let $f : \bar{s} \Longrightarrow s$ with $\bar{\gamma} \leq
  \nu_{\bar{s}}, f$``$\bar{\gamma} \subseteq \gamma \leq \alpha_{\nu}, \bar{q}
  \in J_{\overline{s}}, f^{\ast} ( \bar{q}) = q$, then

  $\tmop{ran} (f^{\ast} f^{\ast}_{( \bar{\gamma}, \bar{q}, \bar{s})})
  \subseteq \tmop{ran} (f^{\ast}_{(\gamma, q, s)}) .$

  With (i) this implies: if $\beta (f) \geq \gamma$ then $f f_{(
  \bar{\gamma}, \bar{q}, \bar{s})} = \text{$f_{(\gamma, q, s)}$}$.

  (v) Set $g = \text{$f_{(\gamma, q, s)}$}$; $\lambda = \lambda (g)$ and $g_0
  = \tmop{red} (g)$. Then $q \in J_{s| \lambda}$ and $g_0 = \text{$f_{(\gamma,
  q, s| \lambda)}$}$.
\end{lemma}

\nod{\bf{Proof: {\noindent}}}{\tmem{(i)}} -{\tmem{(iv)}} are easy
consequences of the definitions. (For {\tmem{(i)}} note this makes
sense since we have specified in effect that $\tmop{ran} (g^{\ast})
\supseteq \tmop{ran} ( f_{(\gamma, q, s)})$.) We establish{\tmem{
(v)}}. We know that $g_0 \Longrightarrow s| \lambda$. \ \ Set $g_0'
= f_{(\gamma, q, s| \lambda)}$ and we shall argue that $g_0 = g_0'$.
Let $k = g_0^{- 1} g_0'$. The argument of Lemma \ref{initseg} shows
that $d (g_0) = d (g)$; as $g_0 \upharpoonright \gamma = \tmop{id}
\upharpoonright \gamma$, and $q \in \tmop{ran} (g_0)$ by
{\tmem{(i)}} the minimality of $g_0' \Longrightarrow s| \lambda$
implies we have such a $k$ defined. Thus $k \in \mathbb{F}$. But $k
\Longrightarrow d (g_0)$ so we conclude, as $d (g_0) = d (g)$, \
that $g k \in \mathbb{F}$. But $\tmop{ran} ((g k)^{\ast}) \cap
\lambda = \tmop{ran} (g^{\ast}) \cap \lambda$. So, using that $g k
\upharpoonright \gamma = \tmop{id} \upharpoonright \gamma$, and $q,
p (s) \in \tmop{ran} (g k)$, and then {\tmem{(i)}} again, we have
$(g k)^{- 1} g = k^{- 1} \in \mathbb{F}$. Hence $k = \tmop{id}_{d
(g'_0)}$ and thus $g_0 = g_0'$.
\mbox{}\hfill Q.E.D. \\

Our definitions are preserved through $\Longrightarrow$ when a map
$f$ is {\tmem{cofinal}}, meaning that $|f|$ is cofinal into $r (f)$:

\begin{lemma}
  \label{lambdapreserve}Let $f : \bar{s} \Longrightarrow s $ with $\lambda (f)
  = \nu$. Set $\bar{\nu} = \nu_{\bar{s}}$, $\nu = \nu_s$, and let
  $\bar{\gamma} < \bar{\nu}, \gamma = f ( \bar{\gamma}), \overline{q} \in
  J_{\bar{s}}, f ( \bar{q}) = q$. Set

  $\bar{g} = f_{( \bar{\gamma}, \bar{q}, \bar{s})}$; $g = f_{(\gamma, q, s)}$.
  \ Then

  (i) $\lambda ( \bar{g}) < \bar{\nu} \longleftrightarrow \lambda (g) < \nu$ ;

  (ii) If \ $\lambda ( \bar{g}) < \bar{\nu} $ then $f (\lambda ( \bar{g})) =
  \lambda (g)$ and $f (\beta ( \bar{g})) = \beta (g)$.
\end{lemma}

\nod{\bf{Proof: {\noindent}}} Assume $\text{$\lambda ( \bar{g}) <
\bar{\nu}$.}$ Set $\overline{h} = \tilde{h}_{\bar{s}}$, $\lambda' =
f (\lambda ( \bar{g}))$. The following is $\text{$\Pi^{(n)
M_{\bar{s}}}_1 (\{\lambda ( \bar{g}), \bar{\gamma}, p (
\bar{s})\})$}$:

$\forall x^n \forall \xi^n < \bar{\gamma} \forall i < \omega (x^n =
\overline{h} (i, \langle \xi^n, \overline{q}, p ( \bar{s}) \rangle)
\wedge x^n < \bar{\nu} \longrightarrow x^n < \lambda ( \bar{g}))$;
if $\bar{\nu} = \tmop{On} \cap M_{\bar{s}}$ then we drop the
conjunct $x^n < \bar{\nu}$. Then

$\forall x^n \forall \xi^n < \gamma \forall i < \omega (x^n = \tilde{h}_s (i,
\langle \xi^n, q, p (s) \rangle) \wedge x^n < \nu \longrightarrow x^n <
\lambda')$

{\noindent}as $f$ is $\Pi^{(n)}_1$-preserving. Hence $\lambda' \geq \lambda
(g)$.

{\tmem{Claim 1:}} $\lambda' \leq \lambda (g)$.

As $\lambda ( \bar{g}) < \bar{\nu} $ we have $\omega \rho ( \bar{g})
< \omega \rho_{\bar{s}}$ by Lemma \ref{suprema}. Hence if we set $A
= A^{n, p_{\bar{s}} \upharpoonright n}$, and $\bar{N} = \langle
J^A_{\rho ( \bar{g})}, A \cap J_{\rho ( \bar{g})} \rangle$ we have
that $\overline{N} \in M_{\bar{s}}$ and is an amenable structure,
with \ \ $\lambda ( \bar{g}) = \sup ( \bar{\nu} \cap
h_{\overline{N}} ( \bar{\gamma} \cup \{ \bar{q}, p ( \bar{s}) \cap
\omega \rho_{\bar{s}} \})$.

{\noindent}Applying $f^{\ast}$, and with $N = f ( \overline{N})$, we have
 $\lambda' = \sup (\nu \cap h_N (\gamma
\cup \{q, p (s) \cap \omega \rho_{\nu} \})$.

For amenable structures (such as $N)$ we have a uniform definition
of the canonical $\Sigma_1 (N)$ Skolem function $h_N$. From $\langle
N, A_N \rangle \subseteq \langle M^n_s, A^n_s \rangle$, we have that
$h_N \subseteq h_s$, and thus
$$\lambda' = \sup (\nu \cap h_s
(\gamma \cup \{q, p (s) \cap \omega \rho_s \})) = \sup (\nu \cap
\tilde{h}_s (\gamma \cup \{q, p (s)\})).$$

Thus $\lambda' \leq \lambda (g)$ and  {\tmem{Claim 1}} is finished.

{\tmem{Claim 2}} $f (\beta ( \bar{g})) = \beta (g)$

Let $\beta = f (\beta ( \bar{g}))$; as $\bar{g} = f_{(\beta (
\bar{g}), \bar{q}, \bar{s})}$ we have $\beta ( \bar{g}) \notin
\tmop{ran} ( \overline{g}$). $\beta = f (\beta ( \bar{g})) = f (\sup
\{ \bar{\delta} < \bar{\nu} \mid \bar{\delta} \subseteq \tmop{ran} (
\overline{g}$)\})= $f (\sup \{ \bar{\delta} < \bar{\nu} \mid
\bar{\delta} \subseteq h_{\overline{N}} ( \bar{\delta} \cup \{
\bar{q}, p ( \bar{s}) \cap \omega \rho_{} \})\})$= $\sup \{\delta <
\nu \mid \delta \subseteq h_{\bar{N}} (\delta \cup \{q, p (s) \cap
\omega \rho_s \})$\}. By the above $\beta\leq\sup \{\delta < \nu
\mid \delta \subseteq h_{\nu} (\delta \cup \{q, p (s) \cap \omega
\rho_s \}) = \beta (g)$. \ Suppose however $\beta < \beta (g) .$
Then in $M_s$ we have:

$\forall \beta^n \leq \beta \exists \xi^n < \gamma \exists i < \omega (\beta^n
= \tilde{h}_s (i, \langle \xi, q, p (s) \rangle)$.

{\noindent}However $f$ is $\Sigma_1^{(n)}$-preserving, so this goes down to
$M_{\bar{s}}$ as:

$\forall \bar{\beta}^n \leq \beta ( \bar{g}) \exists \bar{\xi}^n <
\bar{\gamma} \exists i < \omega ( \bar{\beta}^n = \tilde{h}_{\bar{s}} (i,
\langle \bar{\xi}^n, \bar{q}, p ( \bar{s}) \rangle)$.

{\noindent}But this, with $\bar{\beta}^n \leq \beta ( \bar{g})$
implies $\beta ( \bar{g}) \in \tmop{ran} ( \bar{g})$ which is a
contradiction! This finishes {\tmem{Claim 2}} and{\tmem{ (ii)}}.
Finally, just note for \ $(\leftarrow)$ of {\tmem{(i)}} \ as $\rho
(f) = \rho_s$, if $\lambda (g) < \nu$ then by Lemma \ref{suprema}
there is $\eta = f ( \bar{\eta}) < \rho (f)$ with $\tilde{h}_s
(\gamma \cup \{q, p (s)\}) \cap \omega \rho_s \subseteq \eta$. This
$\Pi^{(n)}_1$ statement goes down to $M_{\bar{s}}$ as
$\tilde{h}_{\bar{s}} ( \bar{\gamma} \cup \{ \bar{q}, p ( \bar{s})\})
\cap \omega \rho_{\bar{s}} \subseteq \bar{\eta}$. Hence $\lambda (
\bar{g}) < \lambda$.
\mbox{}\hfill Q.E.D. \\

From this point onwards in the proof we are very much following,
almost verbatim, the development of {\cite{BJW}}: the fine
structural arguments specific to our level of mice have all been
dealt with, and the rest is very much combinatorial reasoning that
is common to whatever model we are trying to define a $\square$
sequence for.

\begin{definition}
\label{Bplus}  Let $s = \langle \nu_s, M_s \rangle \in S^+, q \in
J_{\nu_s} .$ \ $B (q, s)
  =_{\tmop{df}} B^+ (q, s) \backslash \{ \nu_s \}$ where
 $$B^+ (q, s) =_{\tmop{df}} \{\beta (f_{(\gamma, q, s)}) \mid \gamma \leq
  \nu_s \}.$$
\end{definition}

 $B (q, s)$ is thus the set of those $\beta < \nu_s$ so that $\beta =
\beta (f)$ where $f = f_{(\beta, q, s})$.

\begin{lemma}\label{3.18}
  Let $f$ abbreviate $f_{(\gamma, q, s)}$. Assume $q \in J_s$. (i) Suppose
  $\gamma \in B (q, s)^{\ast} .$ Then
$\tmop{ran} (f) = \bigcup_{\beta \in B (q, s) \cap
  \gamma} \tmop{ran} (f_{(\beta, q, s)}) .$

  {\noindent}(ii) Let $\gamma \leq \alpha_s .$ Suppose $\bar{s}$ is such that
  $f : \bar{s} \Longrightarrow s$ with $f ( \bar{q}) = q$. Then $\gamma \cap B
  (q, s) = B ( \bar{q}, \bar{s}) .$

  {\noindent}(iii) Let $\lambda = \lambda (f)$; $f_0 = \tmop{red} (f)$. Then
  $\gamma \cap B (q, s| \lambda) = \gamma \cap B (q, s)$.
\end{lemma}

\nod{\bf{Proof: }}{\tmem{(i)}} is clear; {\tmem{(ii) }}follows from
Lemma \ref{minimality}{\tmem{(iv)}}, and {\tmem{(iii) }}from
{\tmem{(ii)}} and Lemma \ref{minimality}{\tmem{(v)}}. \ \ \
\mbox{}\hfill Q.E.D.

\begin{definition}
  Let $s \in S^+, q \in J_s .$

  \ \ \ \ \ \ \ \ $\Lambda^+ (q, s) =_{\tmop{df}} \{\lambda (f_{(\gamma, q,
  s)}) | \gamma \leq \nu_s \}; \Lambda (q, s) =_{\tmop{df}} \Lambda^+ (q, s)
  \backslash \{ \nu_s \}$.
\end{definition}

The sets  $ \Lambda (q, s) \subseteq C_s $ are first approximations
to $C_s$ if $q$ is allowed to vary. We first  analyse these sets.

\begin{lemma}
  \label{Bigl}Let $s \in S^+, q \in J_s .$ (i) $ \Lambda (q, s)$ is closed
  below $\nu_s$; (ii) $\tmop{ot} ( \Lambda (q, s)) \leq \nu_s$; (iii) if
  $\lambda \in \Lambda (q, s)$ then $q \in J_{s| \lambda}$ and $ \Lambda (q,
  s| \lambda) = \lambda \cap \Lambda (q, s)$.
\end{lemma}

\nod{\bf{Proof:}} Set $\Lambda = \Lambda (q, s)$. {\tmem{(i): }}Let
$\eta \in \Lambda^{\ast}$. We claim that $\eta \in \Lambda^+ (q,
s)$. For each $\lambda \in \Lambda (q, s) \cap \eta$ pick
$\beta_{\lambda} \in B (q, s)$ with $\lambda (f_{(\beta, q, s)}) =
\lambda$. \ Clearly $\lambda \leq \lambda' \longrightarrow
\beta_{\lambda'} \leq \beta_{\lambda}$. Let $\gamma$ be the supremum
of these $\beta_{\lambda}$. As $B (q, s)$ is closed (by {\tmem{(i)}}
of Lemma \ref{3.18}), $\lambda (f_{(\gamma, q, s)}) = \sup_{\lambda}
\lambda (f_{(\beta_{\lambda}, q, s)}) = \eta$.

{\tmem{(ii)}} is obvious; {\tmem{(iii)}}: Let $\lambda \in
\Lambda,$and $g = \lambda (f_{(\gamma, q, s)})$, where we take
$\beta = \beta (g)$. Suppose $g : \bar{s} \Longrightarrow s$. Let $g
( \overline{q}) = q$ and set $ g_0 = \tmop{red} (g)$. Then by Lemma
\ref{minimality}{\tmem{(v)}} $g_0 =$ $\lambda (f_{(\beta, q, s)}))$.
If $\gamma \geq \beta$ then $\lambda = \lambda (f_{(\gamma, q,  s|
\lambda)}) \leq \lambda (f_{(\gamma, q, s)})$. \ If $\gamma \leq
\beta$ then

$|f_{(\gamma, q, s| \lambda)}) | = |g_0 | |f_{(\gamma, \bar{q},
\bar{s})} | = |g| |f_{(\gamma, \bar{q}, \bar{s})}| = |f_{(\gamma, q,
s)} |$

{\noindent}where the first equality is justified by Lemma
\ref{minimality}{\tmem{(v)}}. \ \ \ \ \ \mbox{}\hfill Q.E.D.

\begin{lemma}
  \label{2.2}If $f : \bar{s} \Longrightarrow s$, $\mu = \lambda (f),
  \overline{q} \in J_{\bar{s}}, f ( \bar{q}) = q$, then:

  (i) $\Lambda ( \overline{q}, \bar{s}) = \varnothing \longrightarrow \mu \cap
  \Lambda (q, s) = \varnothing$,

  (ii) $f$``$\Lambda ( \overline{q}, \bar{s}) \subseteq \Lambda (q, s| \mu)$,

  (iii) If $\overline{\lambda} = \max \Lambda ( \overline{q}, \bar{s})$ and
  $\lambda = f ( \bar{\lambda})$ then $\lambda = \max (\mu \cap \Lambda (q,
  s)) .$
\end{lemma}

\nod{\tmstrong{Proof:}} \ {\tmem{(i)}} By its definition, if
$\Lambda ( \overline{q}, \bar{s}) = \varnothing$ then $f_{(0,
\overline{q}, \bar{s})}$ is cofinal into $\bar{\nu}$. Hence
$\tmop{ran} (f \text{$f_{(0, \overline{q}, \bar{s})}$})$ is both
cofinal in $\mu$, and contained in $\tmop{ran} (f_{(0, q, s)})$ by
Lemma \ref{minimality}{\tmem{(iv)}}, thus $\mu \cap \Lambda (q, \nu
s) = \varnothing$. This finishes {\tmem{(i)}}. \ Note that By
\ref{Bigl}{\tmem{(iii)}} $\Lambda (q, s| \mu) = \mu \cap \Lambda (q,
s)$. Let $\text{$f_0 = \tmop{red} (f)$}$.

{\tmem{(ii)}} Let $\overline{\lambda} = \lambda (f_{( \bar{\beta},
\overline{q}, \bar{s})}) \in \Lambda ( \overline{q}, \bar{s})$, and let $f (
\overline{\beta}, \overline{\lambda}) = \beta, \lambda = f_0 (
\overline{\beta}, \overline{\lambda})$. \ Then $f_0 (\lambda (f_{(
\bar{\beta}, \bar{q}, \bar{s})})) = \lambda (f_{(\beta, q,  s| \mu)}) \in
\Lambda (q, s| \mu)$.

{\tmem{(iii)}} Let $\overline{\beta} = \sup \{\gamma | \lambda (
\text{$f_{(\gamma, \bar{q}, \bar{s})}$}) \leq \overline{\lambda}
\}$. Then $\lambda ( \text{$f_{(\bar{\beta},{\bar{q}, \bar{s}})}$})
= \overline{\lambda}$, and by the assumed maximality of
$\overline{\beta}$ we have $\lambda ( \text{$f_{\bar{\beta} + 1,}
{\bar{q}, \bar{s}})$}) = \bar{\nu}$. Set $\beta = f ( \bar{\beta}) =
f_0 ( \bar{\beta})$, then by (IV)(2), $\lambda = f_0 (
\bar{\lambda}) = \lambda (f_{\beta, q, s| \mu})$. However $\lambda
(f_{\beta + 1, q, s| \mu}) \geq \mu$, since, again by Lemma
\ref{minimality}{\tmem{(iv)}}, $\tmop{ran} ( \text{$f_0 f_{(
\bar{\beta} + 1,},{\bar{q}, \bar{s})}) \subseteq \tmop{ran}
(f_{(\beta + 1, q, s| \mu)}) .$}$ Thus $\lambda = \max (\Lambda (q,
s| \mu)) = \max (\mu \cap \Lambda (q, s)) .$  \mbox{}\hfill
Q.E.D. \\

The p.r. definitions of $\lambda (f)$, $B (q, s)$, $ \Lambda (q,
s),$ are uniform in the appropriate parameters. If $s = \langle \mu,
M_{\mu} \rangle \in S^+$, then if we may define $F_{\text{$s$}}
=\{f_{(\gamma, q, s| \nu)} | \nu \in S \cap \mu, q \in J_{s| \nu},
\gamma \leq \nu\}, E_{\text{$s$}} =\{\langle \nu, M_{s| \nu}, p (s|
\nu), \tilde{h}_{s| \nu} \rangle | \nu \in S \cap \mu$\},
$G_{\text{s}} =\{ \langle \langle s| \nu, q \rangle, \Lambda (q, s|
\nu \rangle | q \in J_{s| \nu}, \nu \in S \cap \mu\}$. We then have:

\begin{lemma}
  \label{Edefinable} (i) $E_{\text{$s$}}, F_{\text{$s$}}, G_s$ are uniformly
  $\Delta_1 (J_s)$ for $s \in S^+$; \\ (ii) $\mu' < \mu \Longrightarrow $ $E_{\mu'}$, $F_{\mu'},
  G_{\mu'} \in J_s .$
\end{lemma}

\begin{lemma}
  Let $f : \bar{s} \Longrightarrow s$ with $\overline{q} \in J_{\bar{s}}, f (
  \bar{q}) = q.$ Then

  (i) If $f$ is cofinal then \ $|f| : \langle J_{\bar{s}}, \Lambda ( \bar{q},
  \bar{s}) \rangle \longrightarrow_{\Sigma_1} \langle J_s, \Lambda (q, s)
  \rangle$;

  (ii) Otherwise: \ $|f| : \langle J_{\bar{s}}, \Lambda ( \bar{q}, \bar{s})
  \rangle \longrightarrow_{\Sigma_0} \langle J_s, \Lambda (q, s) \rangle$
\end{lemma}

\nod{\bf{Proof: {\noindent}}}{\tmem{(i) }}It suffices to show that
$|f| (\Lambda ( \bar{q}, \bar{s}) \cap \bar{\tau}) = \Lambda (q, s)
\cap f (\tau)$ for arbitrarily large $\overline{\tau} <
\nu_{\bar{s}}$. However this follows from the last lemma and
\ref{2.2}.

However, if $\bar{\lambda} \in \Lambda ( \bar{q}, \bar{s}),$ then \
$\Lambda ( \bar{q}, \bar{s}) \cap \bar{\lambda} = \Lambda ( \bar{q},
\bar{s} | \bar{\lambda})$ by Lemma \ref{Bigl}, and by the last
lemma, if $f ( \bar{\lambda}) = \lambda$, we have $f (\Lambda (
\bar{q}, \bar{s} | \bar{\lambda})) = \Lambda (q, s| \lambda) =
\lambda \cap \Lambda (q, s)$ (with the latter equality by Lemma
\ref{Bigl} again). If $\Lambda ( \bar{q}, \bar{s})$ is unbounded in
$\nu_{\bar{s}}$, this suffices; if it is empty or bounded, then the
Lemma \ref{2.2}  takes care of these cases. \

For non-cofinal maps {\tmem{(ii) }}we still have, if $\lambda (f) =
\mu,$ that $$|f_0 | : \langle J_{\bar{s}}, \Lambda ( \bar{q},
\bar{s}) \rangle \longrightarrow_{\Sigma_1} \langle J_{s| \mu},
\Lambda (q, s| \mu) \rangle$$ where $f_0 = \tmop{red} (f)$. But
$\Lambda (q, s| \mu) = \mu \cap \Lambda (q, s)$, and $|f_0 | = |f|$.
\ \mbox{}\hfill Q.E.D. \\

The $C_s $ sets may be decomposed into a finite sequence of sets of
the form \ $\Lambda (l^i_s, s)$.

\begin{definition}
  Let $s \in S^+, \eta \leq \nu_s$. $l^i_{\eta s} < \nu_s$ is defined for $i <
  m_{\eta s} \leq \omega$ by induction on $i :$
$$l_{\eta s}^0 =
  0; \ \ \ \ \ \ \ l^{i + 1}_{\eta s} \simeq \max (\eta \cap \Lambda
  (l^i_{\eta s}, s)) .$$

  We also write $l^i$ for $l^i_{\eta s}$ if the context is clear; also we set
  $l^i_s \simeq l^i_{\nu_s s}$; $m_s = m_{\nu_s s}$.
\end{definition}

{\noindent}Some  facts about this definition may be easily checked:

Fact \label{lfact}$\bullet  \text{ $l^i_{\eta s} \leq l^{i +
1}_{\eta s}$ $(i < m_{\eta s})$ is monotone} $ \ \ \ \ \ $\bullet
\,\,i
> 0 \longrightarrow l^i_{\eta s} \in \eta \cap C_s .$

$\bullet $ Let $l^i_{\eta s}$ be defined, and suppose $l^i_{\eta s} < \mu \leq
\eta$. Then $l^i_{\eta s} = l^i_{\mu s} .$

{\noindent}(The last here is by induction on $i$.)

\begin{lemma}
  \label{3.3} Let $f : \bar{s} \Longrightarrow s .$ (i) If $\lambda = \lambda
  (f)$ then $l^i_{\lambda s} \simeq f (l^i_{\bar{s}})$;

  (ii) let $ \overline{\eta} < \nu_{\bar{s}}, f ( \bar{\eta}) = \eta$; then
  $l^i_{\eta s} \simeq f (l^i_{\bar{\eta}  \bar{s}}) .$
\end{lemma}

\nod{\tmstrong{Proof}} {\tmem{(i)}} By induction on $i$. \ If $i =
0$ this is trivial. Suppose $i = j + 1$. Then, as inductive
hypothesis \ $l^j_{\lambda s} = f (l^j_{\bar{s}})$, and thus \ $|f|
: \langle J_{\bar{s}}, \Lambda (l^j_{\bar{s}}, \bar{s}) \rangle
\longrightarrow_{\Sigma_1} \langle J_{s| \lambda}, \Lambda
(l^j_{\lambda s}, s| \lambda) \rangle,$ by the last lemma, as $|
\tmop{red} (f) | = |f|.$ \ However $\Lambda (l^j_{\lambda s}, s|
\lambda) = \lambda \cap \Lambda (l^j_{\lambda s}, s)$, by
\ref{Bigl}. Hence: $f (l_{\bar{s}}^i) \simeq f (\max \Lambda
(l^j_{\bar{s}}, \bar{s})) \simeq \max (\lambda \cap \Lambda
(l^j_{\lambda s}, s)) \simeq l^i_{\lambda s}$ with the middle
equality holding by Lemma \ref{2.2}(iii). \ {\tmem{(ii)}} is proved
similarly.  \mbox{}\hfill Q.E.D.

\begin{corollary}
  \label{3.5}(i) Let $f : \bar{s} \Longrightarrow s$ cofinally. Then $l^i_s
  \simeq f (l^i_{\bar{s}})$.

  (ii) Let $\lambda \in C_s .$ Then $l^i_{\lambda s} \simeq l^i_{s| \lambda}
  .$
\end{corollary}

\nod{\tmstrong{Proof}} {\tmem{(i)}} is immediate. For {\tmem{(ii)}}
choose $f : \bar{s} \Longrightarrow s$ with $\lambda = \lambda (f)$,
and set $f_0 = \tmop{red} (f)$. Then $l^i_{\lambda s} \simeq f
(l^i_{\bar{s}}) \simeq f_0 (l^i_{\bar{s}}) \simeq l^i_{s| \lambda}$
with the last equality holding from {\tmem{(i)}}. \ \ \ \ \ \ \ \ \
\ \ \ \ \ \mbox{}\hfill Q.E.D.

\begin{lemma}
  \label{3.7}Let $\eta \leq \nu, \lambda = \min (C_s^+ \backslash \eta) .$
  Then $l^i_s \simeq l^i_{\lambda s} \simeq l^i_{\eta s} $ (for any $i <
  \omega$ for which either side is defined).
\end{lemma}

\nod{\tmstrong{Proof}} Induction on $i$, again $i = 0$ is trivial.
Suppose  $l^j_s = l^j_{\eta s} = l^j_{\lambda s}$ and $i = j + 1$.
Set $l = l^j_{\eta s}$, then we have: $\Lambda (l, s) \cap \eta =
\Lambda (l, s) \cap \lambda$, since $\Lambda (l, s) \subseteq C_s$
and $C_s \cap [\eta, \lambda) = \varnothing .$ Suppose, without loss
of generality that $l^i_{\eta s}$ is defined. Then $l^i_{\eta s} =
\max (\eta \cap \Lambda (l, s)) = \max (\lambda \cap \Lambda (l, s))
= l^i_{\lambda s} = l^i_{s| \lambda}$. \ \ \ \ \ \mbox{}\hfill
Q.E.D.

\begin{lemma}
  Let $j \leq i < m_s .$ Set $l = l^i_s$. Then $l^j_s \in \tmop{ran} (f_{0, l,
  s})$.
\end{lemma}

\nod{\tmstrong{Proof}} Set $f = f_{(0, l^{}, s)}$. Suppose $f :
\bar{s} \Longrightarrow s$, and $\lambda = \lambda (f)$. Then
$l^j_{\lambda s} \simeq f (l^j_{\bar{s}})$ by Lemma
\ref{3.3}{\tmem{(i).}} But $l^j_s$ exists, and $l^j_s < \lambda \leq
\nu_s$. Hence $l^j_s = l^j_{\lambda s} = f (l^j_{\bar{s}})$. \mbox{}\hfill Q.E.D. \\

Importantly the $\langle l^j_{\lambda s} \rangle$ sequences are
finite.

\begin{lemma}
  Let $s \in S^+, \eta \leq \nu_s$. Then $m_{\eta s} < \omega$.
\end{lemma}

\nod{\tmstrong{Proof}} Suppose this fails. Then for some $\eta \leq
\nu_s$ we have that $l^i_{\eta s}$ is defined for $i < \omega$. Let
$\lambda = \min (C_s^+ \backslash \eta)$. then $l^i_{\lambda s} =
l^i_{\eta s}$ by Lemma \ref{3.7}. Choose $f : \bar{s}
\Longrightarrow s$ with $\lambda = \lambda (f)$. Then $l^i_{\lambda
s} = l^i_{s| \lambda} = f (l^i_{\bar{s}})$ for $i < \omega$ by Cor.
\ref{3.5}{\tmem{(ii)}} \& Lemma \ref{3.3}(i). Taking $\lambda$ for
$\nu_s$, we assume, without loss of generality, that $l^i_s $ is
defined for $i < \omega$ for some $s \in S$.
We obtain an infinite descending chain of ordinals by showing that
as $i$ increases, and with it $l^i_s $, the maximal $\beta^i $ that
must be contained in the range of any $f : \Longrightarrow s$
together with $l^i_s$ in order for $\tmop{ran} (f)$ to be unbounded
in $s$ strictly {\tmem{decreases.}} This is absurd.

Set $\text{$l = l^i_s$} .$ Define: $\beta^i = \beta^i_s =_{\tmop{df}} \max
\{\beta | \lambda (f_{(\beta, l, s)}) < \nu_s \}.$ \ By the definition of
$l^{i + 1}_s$ we have that $\lambda (f_{(\beta, l, s)}) < \nu_s
\longleftrightarrow \lambda (f_{(\beta, l, s)}) \leq l^{i + 1}_s$.
Furthermore, by the definition of $\beta^i :$

(1) $\lambda (f_{(\beta^i, l, s)}) \leq l^{i + 1}_s$;

(2) $ \lambda (f_{(\beta^i + 1, l, s)}) = \nu_s$.

{\tmem{Claim}} $\beta^{i + 1} < \beta^i $ for $i < \omega$.

\nod{\tmstrong{Proof}} Set $f = f_{(\beta^{i + 1}, l^{i + 1}, s)} .$
Then $\lambda (f) = l^{i + 2}$,  dropping the subscript $\nu$. Let
$f : \bar{s} \Longrightarrow s$. Then $l^j_{\bar{s}}$ exists and $f
(l^{j_{}}_{\bar{s}}) = l^{j_{}}_{l^{i + 1}, s} = l^j_s $for $j \leq
i + 1$ since $l^j_{} < l^{i + 1} < \nu_s$ (with the first equality
from Lemma \ref{3.3}{\tmem{(i)}} and (1), the second from Lemma
\ref{3.7}).

(3) $\beta^i \geq \beta^{i + 1} .$

\nod{\tmstrong{Proof}} of (3): Suppose not, then $(\beta^i + 1) \cup
\{l^i \} \subseteq \tmop{ran} (f)$. Hence $\tmop{ran} (f_{(\beta^i +
1, l^i, s)}) \subseteq \tmop{ran} (f)$. hence by (2), $\lambda (f) =
\nu_s > l^{i + 2} .$ Contradiction!

(4) $\beta^i \neq \beta^{i + 1} .$

\nod{\tmstrong{Proof}} of (4): Suppose not. As $\beta^{i + 1}$ is
the first ordinal moved by $f$ we conclude that $f (\beta^i) >
\beta^i .$ Set $g =$ $f_{(\beta^i, l^{}, s)}, \bar{g} = f_{(\beta^i,
\bar{l}, \bar{s})}$ where $\bar{l} = l^i_{\bar{s}}$ . Then $g = f
\bar{g}$, since $f \upharpoonright \beta^i = \tmop{id}, f ( \bar{l})
= l (= l^i_{\bar{s}}) .$ Hence \ $l^{i + 1} = \lambda (g) = \lambda
(f \bar{g}) < l^{i + 2} = \lambda (f)$. Hence $\lambda ( \bar{g}) <
\nu_{\bar{s}}$. Now we set:  $g' = f_{(f (\beta^i), l, \nu)}$ and
$g_0 = f_{(\beta^i, l, s|l^{i + 2})}$. If further $f_0 = \tmop{red}
(f)$, then we have also $g_0 = f_0 \overline{g}$ by
\ref{minimality}{\tmem{(iv)}}. As $l^{i + 1} = \lambda (g) < l^{i +
2}$, Lemma \ref{lambdapreserve}{\tmem{(ii)}} applies and:

\ \ \ \ \ \ \ \ \ \ \ \ \ \ \ \ $f (\beta ( \bar{g})) = \text{$f_0 (\beta (
\bar{g})) = \beta (g_0) = \beta (g) = \beta^i .$}$

{\noindent}Hence $\beta^i \in \tmop{ran} (f)$ which is a
contradiction. This proves the {\tmem{Claim}} and hence the
Lemma.\mbox{}\hfill Q.E.D. \\

We now set $l_{\eta s} = l^{m - 1}_{\eta s}$, where $m = m_{\eta s}
.$ Again we write $l_s$ for $l_{\nu_s s}$. Notice that then $\Lambda
(l_{\eta s}, \nu_s) \cap \eta$ is either unbounded in $\eta$ or is
empty. We first analyze the latter case.

\begin{lemma}
  \label{4.1}Suppose $\Lambda (l_{\eta s}, s) \cap \eta = \varnothing$. Set $l
  = l_{\eta s} .$ Then:

  (i) $l = 0 \longrightarrow C_s \cap \eta = \varnothing,$

  (ii) \ $l > 0 \longrightarrow l = \max (C_s \cap \eta$),

  (iii) $\eta \in C^+_s \longrightarrow \eta = \lambda (f_{(0, l, s)}) .$
\end{lemma}

\nod{\tmstrong{Proof}} \ \ Set $\rho = \min (C^+_s \backslash (l +
1)$.

(1) $l = l_{\rho s} .$

\nod{\bf Proof:} Set $n = m_{\eta s} - 1.$ Then $l = l^n_{\eta s} <
l + 1 < \eta$. Hence (by Fact after \ref{lfact}) $l = l^n_{l + 1,
s}$. But $\Lambda (l, s) \cap (l + 1) = \varnothing .$ Hence $l^{n +
1}_{l + 1, s}$ is undefined and $l = l_{l + 1, s}$. Hence $l =
l_{\rho, s} $ by Lemma \ref{3.7}. \ \mbox{}\hfill
Q.E.D.(1)\\

(2) $\lambda (f_{(0, l, s)}) = \rho$.

\nod{\bf Proof:} Choose $f : \bar{s} \Longrightarrow s$, with
$\lambda (f) = \rho$ witnessing that $\rho \in C_s .$ Then, by Lemma
\ref{3.3}(i), $f (l_{\bar{s}}) = l_{\rho s} = l$. Set $\bar{l} =
l_{\bar{s}}$. Now note that we must have that $\lambda (f_{(0,
\bar{l}, \bar{s})}) = \bar{s}$. For, if this failed then $f (\lambda
(f_{(0, \bar{l}, \bar{s})})) = \lambda (f_{(0, l, s)}) < \rho$ by
Lemma \ref{lambdapreserve} \ and so the latter is in $C^+_s \cap (l,
\rho)$, which is absurd! Then \ $\lambda (f_{(0, l, s)}) = \lambda
(f f_{(0, \bar{l}, \bar{s})}) = \lambda (f) = \rho$.  \mbox{}\hfill Q.E.D.(2)\\

From (2) and the definition of $l$ as $l_{\eta s})$ it follows that
$\rho \geq \eta$ . There are thus three alternatives:

If $l = 0$ then {\tmem{(i)}} holds: $\rho = \min (C^+_s \backslash
1) = \min (C^+_s) \geq \eta$. If $l > 0$ then $l = \max (C_s \cap
\eta)$ since $(C_s \cap \eta) \backslash (l + 1) \subseteq$ $(C_s
\cap \rho) \backslash (l + 1) = \varnothing$ and thus we have
{\tmem{(ii)}}; finally for {\tmem{(iii)}} if $\eta \in C^+_s
\longrightarrow \eta = \max (C^+_s \backslash (l + 1) = \rho =
\lambda (f_{(0, l, s)}) .$ \mbox{}\hfill
Q.E.D. \\

We now get a characterisation of the closed sets $C_s^+$.

\begin{lemma}
  \label{4}Let $\lambda$ be an element or a limit point of $C^+_s$. Let $l =
  l_{\lambda s}$. Then there is $\beta$ such that $\lambda = \lambda
  (f_{(\beta, l, s)})$. Hence $C_s$ is closed in $\nu_s$, and $C^+_s =\{
  \lambda (f_{(\beta, l, s)}) \mid \beta \leq \nu_s, l < \nu_s \}.$
\end{lemma}

\nod{\tmstrong{Proof}} {\tmem{Case 1}} $\lambda \cap \Lambda (l, s)
= \varnothing$

Then $C_s \cap \lambda = \varnothing$ or $l = \max (C_s \cap \lambda)$ by the
last lemma. Hence $\lambda$ is not a limit point of $C^+_s$. Hence $\lambda
\in C^+_s$, and thus $\lambda = \lambda (f_{(0, l, s)})$ by{\tmem{ (iii)}} of
that lemma.

{\tmem{Case 2}} \ $\lambda \cap \Lambda (l, s)$ is unbounded in $\lambda$.

Given $\mu \in \Lambda (l, s) \cap \lambda$, let $\beta_{\mu}$ be
such that $\lambda (f_{(\beta_{\mu}, l, s)}) = \mu$. Then $\lambda
(f_{(\beta, l, s)}) = \lambda$ where $\beta = \sup_{\mu} \beta_{\mu}
.$\\  The last sentence is immediate from the
previous one. \mbox{}\hfill Q.E.D. \\

We remark that we have just shown that the first conjunct of (i) of Theorem
\ref{ksquare} holds. We move towards proving the other clauses. The following
is (iii).

\begin{lemma}
  $\lambda \in C_s \longrightarrow \lambda \cap C_s = C_{s| \lambda} .$
\end{lemma}

\nod{\tmstrong{Proof}} \ Assume inductively the result proven for
all $\nu'$ with $\nu' < \nu_s$ and $s| \nu' \in S$, (that is, the
lemma is proven with $s| \nu'$ replacing $s$) and we prove the lemma
for $\nu_s$ by induction on $\lambda$. Let $l = l_{\lambda s} .$
Hence by Cor.\ref{3.5} $l = l_{s| \lambda} .$ By Lemma \ref{4}
$\lambda \in \Lambda (l, s)$. Set $\Lambda = \lambda \cap \Lambda
(l, s)$. Then by Lemma \ref{Bigl}{\tmem{(ii)}} $\Lambda = \Lambda
(l, s| \lambda)$.

{\tmem{Case 1 $\Lambda = \varnothing$}}.

If $l = 0$, then $C_{s| \lambda} \subseteq \lambda \cap C_s = \varnothing$
(the latter by Lemma \ref{4.1}). If $l > 0$, then $l = l_{s| \lambda} = \max
(C_{s| \lambda} \cap \lambda) = \max (C_{s| \lambda}) = l_{\lambda s} = \max
(\lambda \cap C_s)$ by the same lemma. As $l < \lambda$, we use the inductive
hypothesis on $\lambda$: $l \cap C_s = C_{s|l} = l \cap C_{s| \lambda}$ where
the second equality is the inductive hypothesis taking $\lambda = \nu' <
\nu_s$. \ Hence $C_{s| \lambda} = \lambda \cap C_s = C_{s|l} \cup \{l\}$.

{\tmem{Case 2}} $\Lambda$ is unbounded in $\lambda$.

Then $\mu \in \Lambda \longrightarrow \mu \in C_s \cap C_{s|
\lambda} .$ Hence by the overall inductive hypothesis $C_{s| \mu} =
\mu \cap C_{s| \lambda}$ and (as $\mu < \lambda$) $C_{s| \mu} = \mu
\cap C_s$. Hence $C_{s| \lambda} = \lambda \cap C_s = \bigcup_{\mu
\in \Lambda} C_{s| \mu} .$ \ \ \ \ \mbox{}\hfill Q.E.D. \\

Now (i) of the Theorem follows easily:

\begin{lemma}
  $\sup (C_s) < \nu_s \longrightarrow \tmop{cf} ( \nu_s) = \omega .$
\end{lemma}

{\noindent}{\tmstrong{Proof}} Let $l = \sup (C_s) = l_s$. Then $
\tmop{ran} (f_{(0, l, \nu_s)})$ is countable, and cofinal in $\nu_s
.$\\ \mbox{}\hfill Q.E.D.

\begin{lemma}
  Let $f : \bar{s} \Longrightarrow s$. Then $|f| : \langle J_{\bar{s}},
  C_{\bar{s}} \rangle \longrightarrow_{\Sigma_0} \langle J_s, C_s \rangle$.
\end{lemma}

{\noindent{\bf{Proof: }}It suffices to show that for arbitrarily
large $\tau < \nu_{\bar{s}} $ that $|f| (C_{\bar{s}} \cap \tau) =
C_s \cap |f| (\tau)$. As usual we continue to write ``$f$'' for
``$|f|$''. \ Set $l_{\bar{s}} = \bar{l}$.

{\tmem{Case 1}} $\Lambda ( \bar{l}, \nu_{\bar{s}})$ is unbounded in
$C_{\bar{s}}$.

If $\bar{\lambda} \in C_{\bar{s}} $ and $\lambda = f (
\bar{\lambda})$ then by \ref{2.2} (and \ref{Bigl}) $\lambda \in
\Lambda (f ( \bar{l}), s) \subseteq C_s .$ By Lemma \ref{Edefinable}
we have $E_{\bar{s} | \bar{\lambda}} \in J_{\bar{s}}$ and $f (
E_{\bar{s} | \overline{\lambda}}) = E_{s| \lambda}$. By Lemma
\ref{4.1}{\tmem{}} $C^{}_{\bar{s} | \bar{\lambda}} =\{\lambda
(f_{(0, l, \bar{s})}) < \bar{\lambda} | l < \bar{\lambda} \} \in
J_{\bar{s}}$ and is uniformly $\Sigma_0$ from $E_{\bar{s} |
\bar{\lambda}} $over $J_{\bar{s}}$. Consequently $|f| (C_{\bar{s} |
\overline{\lambda}}) = C_{s| \lambda},$ by $\Sigma_1$-elementarity
of $|f|$. But $C_{\bar{s} | \overline{\lambda}} = \overline{\lambda}
\cap C_{\bar{s} | \bar{\nu}}$, $C_{s| \lambda} = \lambda \cap C_s .$

{\tmem{Case 2}} $\Lambda ( \bar{l}, \bar{\nu}) = \varnothing$.

Let $f ( \bar{l}) = l$. Then $l = l_{\lambda \nu}$ where $\lambda =
\lambda (f)$. However \ $\lambda (f_{0, \bar{l}, \bar{s}}) =
\nu_{\bar{s}}$ by our case hypothesis. \ Thus $\lambda (f_{(0, l,
s)}) = \lambda (f f_{(0, \bar{l}, \bar{s})}) = \lambda$. Hence
$\Lambda (l, \nu) \cap \lambda = \varnothing$.  By Lemma \ref{4.1}
we are reduced to the following two subcases:

{\tmem{Case 2.1}} $\bar{l} = l = 0$. Then, $C_{\bar{s}} = C_s \cap
\lambda = \varnothing$, and so the result is trivial.

{\tmem{Case 2.2}} $\bar{l} = \max C_{\bar{s}} .$ Then $l > 0$ and
thus $l = \max (C_s \cap \lambda)$. Hence for sufficiently large
$\bar{\tau} > \bar{l}$ \ $f ( \bar{\tau} \cap C_{\bar{s}}) = f
(C_{\bar{s}}) = f (C_{\bar{s}} \cap \bar{l} \cup \{ \bar{l} \}) =
(C_s \cap l) \cup \{l\}= C_s \cap \lambda = f ( \bar{\tau}) \cap C_s
.$ \ \ \ \ \mbox{}\hfill Q.E.D.

We now proceed towards calculating the order types of the
$C_s$-sequences. This is done (in a somewhat speedy manner) in
{\cite{BJW}}, but the following comes from {\cite{Je1?}}. We first
generalise the definition of $\beta^i$.

\begin{definition}
  For $\eta \leq \nu_s \tmop{set} : \beta^i_{\eta s} \simeq_{} \max
  \{\beta | \lambda (f_{(\beta, l^i_{\eta s}, s)}) < \eta\}$.
\end{definition}

In very close analogy to the $\beta^i=\beta^i_s$ we have parallel
 properties for the
$\beta^{i_{}}_{\eta s}$:
\begin{enumeratenumeric}
  \item $\lambda (f_{(\beta, l^i_{\eta s}, s)}) < \nu_s \longleftrightarrow
  \lambda (f_{(\beta, l^i_{\eta s}, s)}) \leq l^{i + 1}_{\eta s}$.

  \item $\beta^i_{\eta s}$ is defined if and only if $l^{i + 1}_{\eta s}$ is
  defined - \tmtextit{i.e.}{\hspace{0.25em}} $i + 1 < m_{\eta s}$.

  \item  $\beta^i_{\eta s} \simeq \beta^i_{\lambda s}$ if $\lambda = \min
  (C_s^+ \backslash \eta)$.  $\lambda (f_{(\beta, l^i_{\eta
  s}, s)}) < \eta \longleftrightarrow \lambda (f_{(\beta, l^i_{\eta s}, s)}) <
  \lambda .$

  \item $\beta^{i + 1}_{\eta s} < \beta^i_{\eta s}$ when defined. (By the same
  argument as for $\beta^{i + 1} < \beta^i .$)

\end{enumeratenumeric}
Now we set $b_{\eta} = b_{\eta s} =_{\tmop{df}} \{ \beta^i_{\eta s}
| i + 1 < m_{\eta s} \}$. For $\eta \in C_s$ we then set $d_{\eta} =
d_{\eta s} =_{\tmop{df}} b_{\eta^+ s}$ where $\eta^+ = \min (C_s^+
\backslash (\eta + 1))$. The subscript $s$ on ordinals remains
unaltered throughout the rest of the proof so we shall drop it.
Then we have:\\

\nod 5. Let $\eta \in C_s$, with $ l^i_{\eta^+} < \eta$. Then by
induction on $i$: $l^i_{\eta^+} = l^i_{\eta}$.

\nod 6. Let $\eta \in C_s,$ with $l^i_{\eta^+} < \eta$ then: $$l^{i
+ 1}_{\eta^+} = \eta\mbox{ if }\eta \in \Lambda (l^i_{\eta}, s),
\mbox{ and } = l^{i + 1}_s \mbox{ otherwise }.$$

Proof of 6: $l^i_{\eta^+} = l^i_{\eta}$ by 5. If $\eta \in \Lambda
(l^i_{\eta}, s)$ then $\eta$ is maximal in this set below $\eta^+$. So the
first alternative holds. Note that $i \neq m_{\eta s} - 1$ (otherwise by Lemma
\ref{4} for some $\beta$, $\eta = \lambda (f_{(\beta, l^i_{\eta s}, s)}) \in
\Lambda (l^i_{\eta}, s))$. Thus $l^{i + 1}_{\eta}$ is defined and $l^{i +
1}_{\eta^+}$ must equal this.

\begin{lemma}
  Let $\eta, \mu \in C_s$, with $\eta < \mu$. Then $d_{\eta} <^{\ast} d_{\mu}
  .$
\end{lemma}

\nod{\tmstrong{Proof}} \ Let $\eta^+ = \min (C_s^+ \backslash (\eta
+ 1))$, $\mu^+ = \min (C_s^+ \backslash (\mu + 1))$. Let $i$ be
maximal so that  $l^i_{\mu^+} = l^i_{\eta^+} .$Then $\beta^j_{\mu^+}
= \beta^j_{\eta^+}$ for $j < i$. As \ $l^i_{\mu^+} \leq \eta < \mu$,
we have by 6. above that $l^{i + 1}_{\mu^+}$ is defined and $l^{i +
1}_{\mu^+} = \mu$ or $l^{i + 1}_{\mu}$. Moreover then
$\beta^i_{\mu^+}$ is defined, and by maximality of $i$, $l^{i +
1}_{\eta^+} \neq l^{i + 1}_{\mu^+}$.

{\tmem{Claim}} \ $l^{i + 1}_{\eta^+} < l^{i + 1}_{\mu^+}$.

That $l^{i + 1}_{\mu^+} < \eta^+$ is ruled out: otherwise $l^{i +
1}_{\eta^+} = l^{i + 1}_{\mu^+}$ again). So  $l^{i + 1}_{\eta^+} <
\eta^+ \leq l^{i + 1}_{\mu^+}$. \mbox{}\hfill Q.E.D. {\tmem{Claim.}}

As $\beta^i_{\mu^+}$ is defined, if $_{} \beta^i_{\eta^+} $is
undefined, then we'd be finished. Set $l = l^i_{\mu^+} =
l^i_{\eta^+} .$ Then $\lambda (f_{(\beta^i_{\eta^+}, l, s)}) = l^{i
+ 1}_{\eta^+}$ and $\lambda (f_{(\beta^i_{\mu^+}, l, s)}) = l^{i +
1}_{\mu^+}$. Hence $_{} \beta^i_{\eta^+} < \beta^i_{\mu^+}$ and thus
$d_{\eta} <^{\ast} d_{\mu}$ as required. \ \mbox{}\hfill Q.E.D.

\begin{lemma}
  \label{ordertype}Let $\alpha$ be p.r. closed so that for some $\alpha_0 <
  \alpha$ $\lambda (f_{(\alpha_0, 0, s)}) = \nu_s$. Then $\tmop{ot} (C_s) <
  \alpha$.
\end{lemma}

\nod{\bf{Proof: }} First note that $\tmop{ot} ( \langle [\alpha ]^{<
\omega}, <^{\ast} \rangle) = \alpha$. Let $\alpha_0 < \alpha$ {be}
{such}, {with} {the} {property} \tmop{that} $\lambda (f_{(\alpha_0,
0, s)}) = \nu_s$. Then $\{\beta^i_{\eta s} \mid \eta \leq \nu_s, i +
1 < m_{\eta s} \} \subseteq \alpha_0$. Thus $\tmop{ot} \langle
\{d_{\eta} \mid \eta \in C_s \}, <^{\ast} \rangle \leq \tmop{ot} (
\langle [\alpha ]^{< \omega}, <^{\ast} \rangle) < \alpha$. Thus
$\tmop{ot}
(C_s) < \alpha$.  \mbox{}\hfill Q.E.D. \\

To obtain the requisite $\langle C_{\nu} \mid \nu \in S \rangle$ for
a Global sequence in $K$, we assign the appropriate level $K_{\beta
(\nu)}$ as $M_s$ over which $\nu$ is definably singularised. Then $s
= \langle \nu, K_{\beta (\nu)} \rangle \in S^+$.  \mbox{}\hfill
Q.E.D.(Global $\LBox)$

\section{Obtaining Inner Models with measurable cardinals}

We assume that we have a Global $\square$ sequence $\langle C_{\nu}
| \nu \in S \rangle$ in $K$ constructed as in the last section.  We
 have:

\begin{theorem}
  \label{3.1}Assume $n >3$ and $\{\alpha < \omega_{n } \mid \alpha \in
  \mathrm{Cof}(\omega_{n-2}) \cap K$-$\tmop{Sing}\}$ is, in $V$, stationary below
  $\omega_{n}$. Then
  \[ T_{n } =_{\tmop{df}} \{\beta \in \mathrm{\tmop{Cof}}(\omega_1) \cap
     \omega_{n} \mid \mathrm{\tmop{ot}} (C_{\beta}) \geq \omega_{n - 3} \}
  \]
  is stationary in $\omega_{n }$.
\end{theorem}

\begin{proof}
  Let $C \subseteq \omega_{n }$ be an arbitrary closed and unbounded set in
  $\omega_{n }$. Take \ $\gamma \in C^{\ast} \cap \mathrm{Cof}(\omega_{n-2})$
  with $\gamma$ a $K$-singular; in other words with $C_{\gamma}$ defined. As
  $\tmop{cf} (\gamma) > \omega$, $C_{\gamma}$ is cub in $\gamma$.  Then $C
  \cap C_{\gamma}$ is closed unbounded in $\gamma$ of ordertype $\geq
  \omega_{n-2}$. Take $\beta \in$ $(C \cap C_{\gamma})^{\ast}$ such that
  $cf(\beta) = \omega_1$ and $\mathrm{{ot}} (C \cap C_{\gamma}
  \cap \beta) \geq \omega_{n - 3}$. By the coherency property \ref{Coherence}(c),
  $C_{\beta} = C_{\gamma} \cap \beta$. Thus $\beta \in C \cap
T_{n}
  \neq \emptyset$.
\end{proof}

Note that $(T_n)_{3 < n < \omega}$ as above would be a sequence of
sets to which we could apply the $\tmop{MS}$-principle, if we knew
that they were (in $V$) stationary beneath the relevant $\aleph_n$.
This is what the assumption in the above theorem achieves. The
following is essentially our main Theorem \ref{Main}.

\begin{theorem}
  \label{0-sharp}If $\tmtextrm{\tmop{MS}} ((\aleph_n)_{1 < n < \omega},
  \omega_1)$ holds then there exists $k < \omega$ so that for all $n>k$, \ there
  is $D_n$, closed and unbounded in $\omega_n$, so that $$D_n \cap
  \mathrm{Cof}(\omega_{n - 2}) \subseteq \{\alpha < \omega_{n} \mid  o^K
  (\alpha) \geq \omega_{n - 2} \}.$$
\end{theorem}

\nod{\bf{Proof: }} We suppose not. Then for arbitrarily large $n <
\omega \quad S^0_n =_{\tmop{df}} \{\alpha < \omega_n \mid \alpha \in
{\mathrm{Cof}}(\omega_{n - 2})\, \wedge \,\tmop{Sing}^K (\alpha$)\}
is stationary in $\omega_n$ by appealing to Mitchell's Weak Covering
Lemma for $K$, \ref{WCL}.

We shall define a sequence $(S_n)_{1<n < \omega}$ of stationary
sets. By Theorem \ref{3.1}, for arbitrarily large $n < \omega,$
$T_n$ is stationary in $\omega_n$; for such $n$ (which we shall call
{\tmem{relevant}}) let $S_n = T_n$; for all other $n>1$ take $S_n =
\mathrm{Cof}(\omega_1) \cap \omega_n$.

Define the first-order structure $\mathfrak{A}=
(H^{}_{\omega^{}_{\omega + 1}}, K_{\omega_{\omega + 1}}, \in,
\vartriangleleft, \langle f_n \rangle_{n < \omega}, \cdots)$ with a
wellordering $\vartriangleleft$ of the domain of $\mathfrak{A}$, and
the sequence of finitary functions $f_n$ including a complete set of
skolem functions for $\mathfrak{A}$. The mutual stationarity
property yields some $X \prec H_{\omega^{}_{\omega + 1}}$ such that
$$\{\omega_n \mid n \leq \omega\} \subseteq X, \hspace{0.75em}
\hspace{0.75em} \forall n >2 \hspace{0.15em} (\sup X \cap \omega_n)
\in S_n, \mbox{ and }\omega_2 \subseteq X.$$

(We may assume without loss of generality the latter clause, since a
direct argument shows that all ordinals less than, say, $\omega_k$
may be added to the hull $X$ without increasing the $\sup X \cap
\omega_n $ for any $n > k$. (This goes as follows: let $X_0$ be a
hull that satisfies the MS property and the first two requirements
above: $\{\omega_n \mid n \leq \omega\} \subseteq X_0, \, \forall n
>2 \hspace{0.25em} (\sup X_0 \cap \omega_n) \in S_n$. We now
consider the enlarged hull of $X =_{\tmop{df}} X_0 \cup \omega_k$ in
$\mathfrak{A}$. Let $n > k$. Consider for each $m,$ and each
$\vec{x} \in [X_0]^p$, $\sup \{ f_m ( \vec{\xi}$,
\tmtexttt{$\vec{x}) \cap \omega_n \mid \text{$\vec{\xi} \in
[\omega_k]^l \}$}$} where we have assumed that $f_m$ is $l + p$-ary.
But this is a supremum definable in $X_0$ from $f_n$, $\vec{x}$,
$\omega_n$, and $\omega_k$. \ Hence it is less than sup$(X_0 \cap
\omega_n$). \ By choice of $\langle f_n \rangle$, every $y \in X$ is
of the form $f_m ( \vec{\xi}$, \tmtexttt{$\vec{x})$} so this
suffices.) \

\ \ Let $\pi : ( \bar{H}_{}, \overline{K}, \in, \ldots) \cong (X, K
\cap X, \in, \ldots)$, be the inverse of the transitive collapse,
and $\beta_n =_{\tmop{df}} \pi^{- 1} (\omega_n)$ for $n \leq
\omega$. For each $2 < n < \omega : \beta_n > \aleph_2$ and
$\op{\tmop{cof}} (\beta_n) = \omega_1$. \ Let $\beta^{\ast}_n
=_{\tmop{df}} \sup (\pi$``$\beta_n)$. We now consider the
coiteration of $K$ with $\overline{K} .$ Let $((M_i, \pi_{i, j},
\nu_i)_{i \leq j \leq \theta} \,\, ,\,\, (N_i, \sigma_{i, j},
\nu_i)_{i \leq j \leq \theta})$ be the resulting coiteration of $(K,
\overline{K})$.\\

(1) {\tmem{The first ultrapower on the $K$ side is taken after a truncation.
In fact $\pi_{0, 1} : M^{\ast}_0 \longrightarrow M_1,$}} {\tmem{where $\pi
\neq \tmop{id}$ and $M^{\ast}_0$ is a proper initial segment of}} $K$.

\nod{\bf{Proof: }}Note that $\beta_3$ is a cardinal of
$\overline{H}$, whilst $K_{\beta_3} = \bar{K}_{\beta_3}$ as $X \cap
\omega_3$ is transitive. However $\tmop{cf} (\beta_3) = \omega_1$
and is thus not a true cardinal of $K$ (by the Covering Lemma for
$K$). Hence the first action of the comparison will be a truncation
on the $K$ side to a structure $M^{\ast}_0$ in which $\beta_3$ is a
cardinal., and thence the ultrapower map $\pi_{0, 1}$ as stated.  \mbox{}\hfill Q.E.D.(1)\\

(2) {\tmem{On the $\overline{K} $ side of the coiteration all the
maps $\sigma_{i, j}$ are the identity:\\ $\forall i \leq \theta N_i
= \overline{K} .$}}

\nod{\bf Proof: } Suppose this is false for a contradiction and let
$\iota$be the least index where an ultrapower of $N_{\iota} =
\overline{K}$ is taken by some $E^N_{\nu_{\iota}}$ with critical
point $\kappa_{\iota}$. On the $K$ side let $\zeta$ be least so that
$\mathcal{P}(\kappa_{\iota}) \cap M_{\iota} \| \zeta =
\text{$\mathcal{P}(\kappa_{\iota}) \cap N_{\iota}$} .$ Let us set
$M^{\ast}$ to be this $M_{\iota} \| \zeta$. (Note that no truncation
is taken in the comparison on the $\overline{K}$ side.). Note that
since $M_0^{\ast}$ was a truncate of $K$, we have that thereafter
each $M_i$ is sound above $\kappa_i$ and that $\omega \rho^{n +
1}_{M_i} \leq \kappa_i < \omega \rho^n_{M_i}$ for some $n = n (i)$.

As $E^N_{\nu_{\iota}}$ is a total measure on $N_{\iota} = \overline{K} $ we
have that $\tilde{E} =_{\tmop{df}} E^{K_{}}_{\pi (\nu_{\iota})} = \pi
(E^N_{\nu_{\iota}})$ is a full measure in $K$ with critical point
$\tilde{\kappa} =_{\tmop{df}} \pi (\kappa_{\iota})$.

We apply the measure $E^N_{\nu_{\iota}}$ to $M^{\ast}$ itself and form the
fine structural ultrapower $\widetilde{M} = \tmop{Ult}^{\ast} (M^{\ast},
\text{$E^N_{\nu_{\iota}}$)}$ with map $t : M^{\ast} \longrightarrow
\widetilde{M}$. Note that by the weak amenability of $E^N_{\nu_{\iota}}$,
$\widetilde{M} \cap \mathcal{P}(\kappa_{\iota}) = M^{\ast} \cap
\mathcal{P}(\kappa_{\iota})$, and that $t$ is $\Sigma_0^{(n)}$ and cofinal.

We should like to compare $M^{\ast} $ with $\widetilde{M}$ but for this we
need the following Claim.

{\tmem{Claim 1}} \ $\widetilde{M}$ is normally iterable above
$\kappa_{\iota}$.

Proof: First note:

(i) $M^{\ast}$ and $\overline{K}$ agree up to $\nu_{\iota}$, hence
if $E_{\iota}$ is the extender sequence on $M_{\iota}$ we have that
$\pi \upharpoonright J_{\nu_{\iota}}^{E_{\iota}} :
J_{\nu_{\iota}}^{E_{\iota}} \longrightarrow
J^{E^K}_{\widetilde{\nu}}$ cofinally for $\widetilde{\nu}
=_{\tmop{df}} \sup \pi \text{``} \nu_{\iota} .$

(ii) $\tmop{cf} (\nu_{\iota}) > \omega$ and hence we have a
canonical extension $\pi^{\ast} \supseteq \pi \upharpoonright
J_{\nu_{\iota}}^{E_{\iota}}$ with $\pi^{\ast} : M^{\ast}
\longrightarrow M'$ with \ \ \ $\text{$\omega \rho^{n +
1}_{M^{\ast}} \leq \kappa_{\iota} < \omega \rho^n_{M^{\ast}}$}$
implying that $\omega \rho^{n + 1}_{M'} \leq \tilde{\kappa} < \omega
\rho^n_{M'}$, $M'$ sound above $\tilde{\kappa}$, and $\pi^{\ast}$
$\Sigma_0^{(n)}$ preserving.

Proof: Note  that $\tmop{cf} (\nu_{\iota}) = \tmop{cf}
(\kappa_{\iota}^{+ M_{\iota}}) > \omega$ since otherwise we have
that $\kappa_{\iota}^{+ M_{\iota}}$ is a $\overline{K}$ cardinal,
which $H$ will think, by Weak Covering, has uncountable cofinality
equal to some $\beta_i$. As $\tmop{cf} (\beta_i) = \omega_1$ it
would be a contradiction to have $\tmop{cf} (\nu_{\iota}) = \omega$.
By the definition of $\zeta$ we have that $\text{$\text{$\omega
\rho^{n + 1}_{M^{\ast}} \leq \kappa_{\iota} < \omega
\rho^n_{M^{\ast}}$}$}$ for some $n$ and that $M^{\ast}$ is sound
above $\kappa_{\iota}$. Consequently $\nu_{\iota}$ is definably
singularized over $M^{\ast}$ and we have the right conditions to
apply \ref{kpseudo} with the other properties mentioned following
from that.  \mbox{}\hfill Q.E.D.(ii)\\

(iii) $\tilde{\kappa}$ a $K$-cardinal, \ $\omega \rho^{n + 1}_{M'} \leq
\tilde{\kappa}$, and $M'$ sound above $\kappa'$ imply that $M'$ is an initial
segment of $K$.

Applying the full measure $\widetilde{E}$ yields $\sigma : K
\longrightarrow_{\tilde{E}} \widetilde{K}$. Let $\widetilde{M}' =
\sigma (M')$, and this is also an initial segment of
$\widetilde{K}$. As $\pi^{\ast} \supseteq \pi \upharpoonright
J_{\nu_{\iota}}^{E_{\iota}}$ we have:

(iv) $X \in \text{$E^N_{\nu_{\iota}} \longleftrightarrow \pi^{\ast} (X) = \pi
(X) \in \widetilde{E}$}$.

Defining $\mathbb{D}(M^{\ast}, E^N_{\nu_{\iota}})$ the term model
for the ultrapower we have:

(v) (a) The map $d ([f]) = \sigma \circ \pi^{\ast} (f) (
\tilde{\kappa})$ is a structure preserving map $d : \text{
$\mathbb{D}(M^{\ast}, E^N_{\nu_{\iota}}) \longrightarrow
\widetilde{M}'$}$. \ \ (a) The map $k : \widetilde{M}
\longrightarrow \widetilde{M}'$ is $\Sigma_0^{(n)}$-preserving with
$k (\kappa_{\iota}) = \tilde{\kappa}$.

Proof: This is a standard computation for (a), and for (b) note that
$\omega \rho^{n + 1}_{\tilde{M}'} \leq \sigma ( \tilde{\kappa}) <
\omega \rho^n_{\widetilde{M}' }$ by (ii) and the elementarity of
$\sigma$. \ \mbox{}\hfill Q.E.D.(v)\\

By (v)(b) since $\widetilde{M}'$ is normally iterable above
$\tilde{\kappa}$ $\widetilde{M}$ will be normally iterable above
$\kappa_{\iota}$, as required. \mbox{}\hfill  Q.E.D. {\tmem{Claim
1}}.

{\tmem{Claim 2}} $E^N_{\nu_{\iota}} = E^{M^{\ast}}_{\nu_{\iota}}$.

Proof: Since $M^{\ast}$ and $\widetilde{M}$ agree up to
$\nu_{\iota}$ the coiteration of these two is above
$\kappa_{\iota}$. \ By {\tmem{Claim 1}} this coiteration is
successful with iterations $i : \widetilde{M} \longrightarrow
\widetilde{M}_{\theta}$ and $j : M^{\ast} \longrightarrow
M^{\ast}_{\theta}$ say.

(vi) The iteration $i$ of $\widetilde{M}$ is above
$(\kappa_{\iota}^+)^{\widetilde{M}} = (\kappa_{\iota}^+)^{M^{\ast}} .$

Proof: $\overline{K}, M^{\ast}, \widetilde{M} $all agree up to
$\nu_{\iota}$ and forming $\widetilde{W} = \tmop{Ult}
(J^{E^{M^{\ast}}}_{\nu_{\iota}}, E^N_{\nu_{\iota}})$ we see
therefore that it is an initial segment of $\widetilde{M}$. \ From
coherence of our extender sequences we know that
$$E^{\widetilde{M}}
\upharpoonright \nu_{\iota} = {E^{\overline{K}} \upharpoonright
\nu_{\iota} = E^{M^{\ast}} \upharpoonright \nu_{\iota}} \ \mbox{ and
}{E_{\nu_{\iota}}^{\widetilde{M} } = {\O} =
E_{\nu_{\iota}}^{\widetilde{W}}}.$$ By the initial segment property
of extender sequences we have that there are no further extenders on
the $E^{\widetilde{M}}$ sequence with critical point
$\kappa_{\iota}$. Hence all critical points used in forming the
iteration map $i$ are above
$(\kappa_{\iota}^+)^{\widetilde{M}}$.  \mbox{}\hfill Q.E.D. (vi)\\

The rest of the argument is fairly standard.

(vii) $\widetilde{M}_{\theta} = M^{\ast}_{\theta}$.

Proof: Let $A \in \Sigma_1^{(n)} (M^{\ast})$ in $p_{M^{\ast}}$ be
such that $A \cap \kappa_{\iota} \notin M^{\ast}$, \ and then note
that $A \cap \kappa_{\iota} \notin \tilde{M} $ as they agree about
subsets of $\kappa_{\iota}$. Hence if the iteration $j$ is simple,
then $M^{\ast}_{\theta}$ is not a proper initial segment of
$\text{$\widetilde{M}_{\theta}$}$. But if $j$ is non-simple then we
reach the same conclusion as no proper initial segment of
$\widetilde{M}_{\theta}$ can be unsound. \ Hence
$\text{$\widetilde{M}_{\theta}$}$ is an initial segment of
$M^{\ast}_{\theta}$. But again we cannot have that it is a proper
initial segment, since using the $\Sigma_0^{(n)}$ preservation
property of $t$ we'd have $A \cap \kappa_i$ in $M^{\ast}_{\theta}$ a
contradiction as before. \ \ \ \mbox{}\hfill Q.E.D. (vii)\\

(viii) (i) $\omega \rho^{n + 1}_{\widetilde{M}} = \omega \rho^{n +
1}_{M^{\ast}} = \omega \rho^{n + 1}_{M_{\theta}^{\ast}}$.

\ \ \ \ (ii) If $p = p_{M^{\ast}} \backslash \omega \rho^{n + 1}_{M^{\ast}}$
then $i \circ t (p) = p_{M^{\ast}_{\theta}, n + 1}$.

\ \ \ \ (iii) $t$ is $ \Sigma^{\ast}$-preserving.

Proof: These are standard arguments from the proof of solidity for
mice - \tmtextit{cf. }{\hspace{0.25em}}{\cite{Z02}} p153-4. In (ii)
one first sees that $i \circ t (p) \in P^{n +
1}_{M^{\ast}_{\theta}}$; a solidity argument on witnesses
$W_{M^{\ast}}^{\alpha, p}$ shows that in fact $\text{$i \circ t (p)
= p_{M^{\ast}_{\theta}, n + 1}$}$.

(ix) $j \upharpoonright \kappa = \tmop{id} = i \circ t \upharpoonright
\kappa$; however $\tmop{crit} (j) = \kappa_{\iota}$.

Proof: As the first clause is immediate, we argue that $j
(\kappa_{\iota}) > \kappa_{\iota}$. As $j$ is an iteration map $j
(p) \in P_{M^{\ast}_{\theta}}^{n + 1}$. By the Dodd-Jensen Lemma
(\em cf.}
{\cite{Z02}} Theorem 4.3.9) $j (p)
\leq^{\ast} i \circ t (p)$, and hence by (8)(ii) we have $j (p) = i
\circ t (p)$. By the soundness of above $\kappa_{\iota}$ we have
that $\kappa = \widetilde{h}_{M^{\ast}}^{n + 1} (i, \xi, p)$ for
some $i < \omega$, some $\xi < \kappa_{\iota}$. Hence $j
(\kappa_{\iota}) = \widetilde{h}_{M_{\theta}^{\ast}}^{n + 1} (i,
\xi, j (p))$. As $j (p) = i \circ t (p)$ we have $j (\kappa_{\iota})
= i \circ t ( \widetilde{h}_{M^{\ast}}^{n + 1} (i, \xi, p)) = i
\circ t (\kappa_{\iota}) > \kappa_{\iota}$. \mbox{}\hfill  Q.E.D.
(ix)\\

Hence $\kappa_{\iota}$ is the first point moved by $j$ and thus some
measure $E^{M^{\ast}}_{\gamma}$ is applied as the first ultrapower
on the $M^{\ast}$ side of the coiteration with $\tmop{crit}
(E^{M^{\ast}}_{\gamma}) = \kappa_{\iota}$ and $\gamma$ least with
$E^{M^{\ast}}_{\gamma} \neq E^{\widetilde{M}}_{\gamma}$. As
$E^{M^{\ast}} \upharpoonright \nu_{\iota} = E^{\widetilde{M}}
\upharpoonright \nu_{\iota}$ and (see the proof of (vi))
$E^{\widetilde{M}}_{\nu_{\iota}} = {\O}$ we must have $\gamma =
\nu_{\iota}$ here. But then
$$X \in E^{M^{\ast}}_{\nu_{\iota}} \longleftrightarrow \kappa_{\iota}
\in j (X) \longleftrightarrow \kappa_{\iota} \in i \circ t (X)
\longleftrightarrow \kappa_{\iota} \in t (X) \longleftrightarrow X
\in E^N_{\nu_{\iota}}.$$

Hence $E^N_{\nu_{\iota}} = E^{M^{\ast}}_{\nu_{\iota}}$ which is our
{\tmem{Claim 2}}. \mbox{}\hfill  Q.E.D. (2)\\

At the $\theta$'th stage therefore, $M_{\theta}$ is an end extension
of $\overline{K} .$ For $n < \omega$, let $i_n$ be the least stage
$i$ where $\kappa_i \geq \beta_n$ if such an $i$ exists, otherwise
set $i_n = \theta$. Let $k_0 < \omega$ be the least $k$ such that
any truncations performed on the $K$ iteration have been performed
before stage $i_k$. We may also assume that from this point
$i_{k_0}$ on then, that the least $m>0$ with
$\omega\rho^m_{M_{\iota}}< \kappa_{\iota}$ is fixed for all $\iota
\geq i_{k_0}$; for this $m$ then, we set $\rho =
\omega\rho^m_{M_{\iota}}$ for any $\iota \geq i_{k_0}$, and we shall
have that any $M_i$ is sound above $\kappa_i$ for $\iota \geq
i_{k_0}$, and thus that $M_\iota = \widetilde{h^{}}^m_{M_\iota}
(\kappa_\iota \cup \{p_{M_\iota} \})$. Further by choice of $m$ note
that for $n>k_0$, $\rho^{m - 1}_{M_{i_n}} > \kappa_{i_n}\geq
\beta_n$. As we have in the iteration that $\pi_{i, j} (\langle
d_{M_{i}},p_{M_i}\rangle) = \langle d_{M_{j}},p_{M_j}\rangle$, and
parameters are finite sequences, we may further assume that $k_0$
has also been chosen sufficiently large so that for any $n \geq
k_0$: (i) $d_{M_{i_n}},p_{M_{i_n}} \cap [\beta_{n - 1}, \beta_n) =
{\O}$, (ii) $k_0$ is itself
relevant.\\

(3) {\tmem{Suppose }}$\langle \kappa_i |i < i_n \rangle$ {\tmem{is unbounded
in}} $\beta_n$, {\tmem{where $n$ is relevant}}. {\tmem{Then for no $i_0 < i_n$
do we have $\pi_{i_0, i} (\kappa_{i_0}) = \kappa_i$}} {\tmem{for unboundedly
many}} $\kappa_i < \kappa_{i_n} .$

Proof: If the conclusion failed then we should have $\pi_{i, j}
(\kappa_{i_{}}) = \kappa_j$ for an $\omega_1$-sub-sequence of the
 sequence of critical points $\langle \kappa_i |i < i_n \rangle$; let us
choose such an $\omega_1$-sub-sequence, and call the set of its
elements $\overline{D}$
with the choice of $\overline{D}$ ensuring that $\overline{D}$ is
closed below $\beta_n$. These are all inaccessible in $\overline{K}
.$ Applying $\pi$, if we set $D = \pi$``$\overline{D}$, then we have
that $D$ is a cub set of order type $\omega_1$ below
$\beta_n^{\ast}$ of $K$-inaccessibles. Note that $\pi$ is continuous
on $\overline{D}$ since $\overline{H}$ is correct about whether any
ordinal $\alpha$ has cofinality $\omega$ or not, since all the
$\beta_n (n < \omega)$ have uncountable cofinality; hence, easily,
if $\kappa_{\lambda}$ is a limit point of $\overline{D}$, then it
has cofinality $\omega$ in $\overline{H}$. If $f : \omega
\longrightarrow \kappa_{\lambda}$ is the least function in
$\overline{H}$ witnessing this, then $\pi (\kappa_{\lambda}) = \pi
(\sup \{f$``$\omega\}) = \sup \{\pi (f (n)) | n \in \omega\}$. (We
are using here that the MS property is formulated using {\tmem{all}}
the $\aleph_n$'s and not just a subsequence.) \ But $n$ is relevant
so $\beta_n^{\ast}$ is singular in $K,$ but of uncountable
cofinality. Thus the closed $C_{\beta_n^{\ast}}$ sequence of $K$ of
$K$-singular ordinals, has non-empty intersection with $D$, which is
absurd.\\

(4) \ {\tmem{If $n \geq k_0$ is relevant then}} $\beta_n$ {\tmem{is}}
$\Sigma_1^{(m)}$ {\tmem{singularised over $M_{i_n}$}} {\tmem{and the latter is
sound above}} $\beta_n .$

Proof: The last conjunct follows from the definition of $i_n : M_{i_n}$ is
sound above $\text{$\widetilde{\kappa} =_{\tmop{df}}$ sup$\langle \kappa_i |i
< i_n \rangle$}$ Divide into the two cases of $\tilde{\kappa} < \beta_n$ or
$\tilde{\kappa} = \beta_{n.} $ In the first case then $M_{i_n} =
\widetilde{h^{}}^m_{M_{i_n}} ( \widetilde{\kappa} \cup \{p_{M_{i_n}} \})$ and
hence $\beta_n$ is so singularised over $M_{i_n}$; in the second case \ take
$\delta < \beta_n$, $\delta \geq \omega \rho_{M_{i_n}}$. Take $i$ minimal such
that $\kappa_i \in [\delta, \beta_n)$. Then $M_i =
\widetilde{h^{}}^m_{M_{i_{}}} (\delta \cup \{p_{M_{i_{}}} \})$ and in
particular $\kappa_i \in \text{$\widetilde{h^{}}^m_{M_{i_{}}} (\delta \cup
\{p_{M_{i_{}}} \})$}$. By (3) take $\gamma < \beta_n$ such that whenever
$\kappa_j \in (\gamma, \beta_n)$ then $\kappa_j \neq \pi_{i j} (\kappa_i)$,
and take some index $j$ such that $\kappa_j \in (\gamma, \beta_n)$. By
elementarity, $\pi_{i j} (\kappa_i) \in \widetilde{h^{}}^m_{M_{j_{}}} (\delta
\cup \{p_{M_{j_{}}} \})$. Since $\kappa_j > \pi_{i j} (\kappa_i)$, the point
$\pi_{i j} (\kappa_i)$ is not moved in the further iteration past stage $j$,
and so $\pi_{i j} (\kappa_i) \in \widetilde{h^{}}^m_{M_{i_n}} (\delta \cup
\{p_{M_{i_n}} \})$. We thus have that
$$(\ast) \quad \rho > \alpha_{\beta_n} =_{\tmop{df}} \max \{\alpha
\mid \sup(\widetilde{h^{}}^m_{M_{i_n}} (\alpha \cup \{p_{M_{i_n}}
\}) \cap \beta_n) = \alpha\}.$$
 But now there must be some $\gamma <
\beta_n$ with sup($\widetilde{h^{}}^m_{M_{i_n}} (\gamma \cup
\{p_{M_{i_n}} \}) \cap \beta_n) = \beta_n$. Because if this failed
we could choose a sequence
$$\gamma_0 = \rho, \gamma_{i + 1} =
\sup(\widetilde{h^{}}^m_{M_{i_n}} (\gamma_i \cup \{p_{M_{i_n}} \})
\cap \beta_n) < \beta_n, \mbox{ and take }\gamma = \sup_i
\gamma_i.$$ As $\tmop{cf} (\beta_n) > \omega$, $\gamma < \beta_n$.
However we have then that $$\gamma =
sup(\widetilde{h^{}}^m_{M_{i_n}} (\gamma \cup \{p_{M_{i_n}} \}) \cap
\beta_n) < \beta_n $$ and simultaneously $\gamma >
\alpha_{\beta_n}$. Contradiction!
(4) is thus proven. \mbox{}\hfill  Q.E.D.(4) \\

(5) {\tmem{If $n$ is relevant, then in the notation of \tmtextup{(4)}, if $m >
1$ then for no smaller $m' < m$ is $\beta_n$ $\Sigma_1^{(m' - 1)}$
singularised over}} $M_{i_n}$.

Proof: \ Just note that as $\rho^{m' - 1}_{M_{i_n}} \geq \rho^{m -
1}_{M_{i_n}} > \beta_n$, any purported $\Sigma_1^{(m' -
1)}$-singularisation over $M_{i_n}$ yields a cofinalising function
{\tmem{in}} $M_{i_n}$. This is absurd as $\beta_n$ is regular in
$M_{i_n}$. \ \ \mbox{}\hfill  Q.E.D.(5) \\

We thus have, by (4), that for relevant $n,$ $s_n =_{\tmop{df}}
\langle \beta_n, M_{i_n} \rangle \in S^+ .$ We therefore have
$C_{s_n}$ sequences associated to such $s_n$ as in the Global
$\LBox$ proof of the previous section.\\

(6){\tmem{ For relevant $n \geq k_0,$ we have}} $\tmop{ot} (C_{s_n})
\leq \widetilde{\beta}$ {\tmem{where }}$\widetilde{\beta}$ {\tmem{is
the least p.r.closed ordinal above}} $\beta_{k_0}$.

Proof: Set $i = i_{k_0}$; $j = i_n$. Then by the usual property of ultrapowers
\ $\pi_{i, j}$``$\omega \rho^{m - 1}_{M_i}$ \ is cofinal in $\omega \rho^{m -
1}_{M_j}$.

Set $s = s_{k_0}$ and let $\delta = \delta_{k_0}$ be least such that
$\lambda (f_{(\delta, 0, s)}) = \beta_{k_0} (= \nu_s)$ where
$f_{(\delta, 0, s)} \Longrightarrow s.$ Then $\delta < \beta_{k_0}$.
Let $Y =_{\tmop{df}} \pi_{i, j} \text{``} \tmop{ran} (
f^{\ast}_{(\delta, 0, s_{_{}})})$. As $\tmop{ran} (
f^{\ast}_{(\delta, 0, s_{})})$ is a $\Sigma_1^{(m - 1)}$ hull in
$M_{s_{}} (= M_i)$ we have that $Y$ is a $\Sigma_1^{(m - 1)}$ hull
in $M_j (= M_{s_n}) .$ We note that $\alpha_{s_{}}, \alpha_{s_n}$
(in the sense of Definition \ref{squareobjects} f)) are below $\rho$
by $(\ast)$ of (4). \ Consequently if we define $\widetilde{Y}
=_{\tmop{df}} \tmop{ran} (f^{\ast}_{(\beta_{k_0} + 1, 0, s_n)})$
then $\widetilde{Y}$ is a $\Sigma_1^{(m - 1)}$ hull of $M_j$.
However $\widetilde{Y} \supseteq Y$, as $\pi_{i, j} (p_s, d_s) =
p_{s_n}, d_{s_n}$, $\pi_{i, j}$ is $\Sigma_1^{(m - 1)}$-preserving,
and $\pi_{i, j} \upharpoonright \beta_{k_0} = \tmop{id}$. (We need
Lemma \ref{dpreserve} here on the preservation of the $d_s$
parameters under iteration.)

By choice of $\delta$ and Lemma \ref{suprema} $\rho (f_{(\delta, 0,
s)}) = \omega \rho_s$. Hence $Y$ is cofinal in $\omega \rho_{s_n}$.
However then $\widetilde{Y}$ is also so cofinal. That is $\rho
(f_{(\beta_{k_0} + 1, 0, s_n)}) = \omega \rho_{s_n}$ which again by
Lemma \ref{suprema} implies $\lambda (f_{(\beta_{k_0} + 1, 0, s_n)})
= \nu_{s_n} = \beta_n$. By Lemma \ref{ordertype} this implies
$\tmop{ot} (C_{s_n}) \leq \widetilde{\beta}$.
 \mbox{}\hfill  Q.E.D.(6)\\

For relevant $n$ we form the ``lift-up'' map $\pi_n^{\ast} : M_{i_n}
\longrightarrow M^{\ast}_n$ which extends $\pi \upharpoonright (
\overline{K} | \beta_n^+$) (where $\beta_n^+ =
(\beta_n^+)^{\overline{K}}$). We obtain the structure $M^{\ast}_n$
and the map $\pi_n^{\ast}$ as a pseudo-ultrapower.\\

(7)(a){\tmem{ For relevant $n$, $\pi_n^{\ast}$ is $\Sigma_1^{(m -
1)}$-preserving, and $\beta_n^{\ast}$ is $\Sigma_1^{(m - 1)}$-singularised
over $M^{\ast}_n$; further, if $m > 1$, then for no smaller $m' < m,$ is
$\beta_n^{\ast}$ is $\Sigma_1^{(m' - 1)}$-singularised over}} $M^{\ast}_n$.

\ \ \ (b) $M^{\ast}_n$ \ \tmtextit{is normally iterable above}
$\beta_n^{\ast} .$

Proof : (a) The Pseudo-Ultrapower Theorem \ref{kpseudo} (with $k = m
- 1)$ shows the right degree of elementarity of $\pi_n^{\ast}$,
\tmtextit{i.e.}{\hspace{0.25em}} that it is $\Sigma_0^{(m - 1)}$
preserving. It further states that the map is cofinal and thus
$\Sigma_1^{(m - 1)}$-preserving, and that it yields that
$\beta_n^{\ast}$ is $\Sigma_1^{(m - 1)}$-singularised over
$M^{\ast}_n$, whilst $\beta_n^{\ast}$ is $\Sigma_1^{(m' -
1)}$-regular over $M^{\ast}_n$ for any $m' < m$ (if $m > 1)$. For
(b) this is a standard argument about canonical extensions defined
from pseudo-ultrapowers using the fact that $\tmop{cf} ( \beta_n^{})
= \tmop{cf} (( \beta_n^{})^{+ M_n}) > \omega$. \ (Note $\tmop{cf}
(\beta_n^+)^{\overline{K}} = \omega_1$, either because
$(\beta_n^+)^{\overline{K}} = (\beta_n^+)^{\overline{H}} = \beta_{n
+ 1}$, or otherwise by applying the Weak Covering Lemma inside
$\overline{H} : \overline{H} \models$``$\tmop{cf}
(\beta_n^+)^{\overline{K}} = \beta_n$'' , and $\beta_n$ of course
has cofinality $\omega_1$.) \ See {\cite{Z02}} \ Lemma 5.6.5.
\mbox{}\hfill  Q.E.D.(7)\\

(8) $M^{\ast}_n$ \ \tmtextit{is an initial segment of} $K.$ \

Proof: Note that by construction $M^{\ast}_n \upharpoonright
\beta_n^{\ast} = K \upharpoonright \beta_n^{\ast} .$ By 7(i)
$\rho_{M^{\ast}_n}^m \leq \beta_n^{\ast}$; again the
pseudo-ultrapower construction shows $M^{\ast}_n$ is sound above
$\beta_n^{\ast}$ and hence is coded by a $\Sigma_1^{(m - 1)}
(M^{\ast}_n)$ subset of $\beta_n^{\ast}$, $A$ say. An elementary
iteration and comparison argument shows that, when $K$ is compared
with $M^{\ast}_n$, to models $N_{\eta}, M^{\ast}_{\eta}$ then $A$ is
$\Sigma_1^{(m - 1)}$ definable over $N_{\eta}$, and thus is in $K$
itself. As $M^{\ast}_n$ is a mouse in $K$, its soundness above
$\beta_n^{\ast}$ implies that after any supposedly necessary
coiteration, we must have $N_{\eta} = M^{\ast}_{\eta} $ and hence
$\tmop{core} (N_1) = \tmop{core} ( N_{\eta}) = \tmop{core}
(M^{\ast}_{\eta}) = M^{\ast}_n$. Hence $M^{\ast}_n$ is an initial
segment of $K$. \mbox{}\hfill  Q.E.D.(8)\\

(9)(a) $s^{\ast}_n = \langle  \beta_n^{\ast}, M^{\ast}_n \rangle \in S^+$;

\ \ (b) $M^{\ast}_n$ \ \tmtextit{is the assigned $K$-singularising structure
for $\beta_n^{\ast}$; hence} \tmtextit{in $K$, $C_{\beta_n^{\ast}}$ is defined
over $M^{\ast}_n$, that is} $C_{\beta_n^{\ast}} =_{\tmop{df}} C_{s^{\ast}_n}$.

Proof: For (a), by (7)(a) $M_n^{\ast}$ singularises appropriately,
it is sound above $\beta_n^{\ast}$, and by (8) it is a mouse. For
(b) we have shown that $M^{\ast}_n$ is an initial segment of $K,$
and thus conforms to the definition of the segment chosen to define
the canonical $C$-sequence associated to $\beta_n^{\ast}$ in $K$. \mbox{}\hfill Q.E.D.(9)\\
We thus conclude:\\

(10) \tmtextit{For relevant} $n \geq k_0$ $\tmop{ot} (C_{\beta_n^{\ast}}) \leq
\pi ( \widetilde{\beta}) < \pi (\beta_{k_0 + 1}) = \omega_{k_0 + 1}$.

Proof: By (6) $\tmop{ot} (C_{s_n}) \leq \widetilde{\beta}$ because $
\tilde{h}_{s_n} (\beta_{k_0} + 1, p (s_n)) $ is cofinal in $ \omega
\rho_{s_n} = \omega \rho^{m - 1}_{M_{i_n}}$. Set $\beta' =
\pi^{\ast}_n (\beta_{k_0} + 1) .$By the $\Sigma_1^{(m -
1)}$-elementarity of $\pi^{\ast}_n$ we shall have that
$\pi^{\ast}_n$``$\tilde{h}_{s_n} (\beta_{k_0} + 1, p (s_n)_{})
\subset \tilde{h}_{s^{\ast}_n} (\beta', p (s^{\ast}_n))$. As
$\pi^{\ast}_n \upharpoonright \omega \rho_{s_n}$ is cofinal into
$\omega \rho_{s^{\ast}_n}$, we deduce that $\rho (f_{(\beta', 0,
s_n^{\ast})}) = \text{$\omega \rho_{s^{\ast}_n}$.}$ By Lemma
\ref{suprema} this ensures that $\lambda (f_{(\beta', 0,
s_n^{\ast})}) = \nu_{s^{\ast}_n} = \beta^{\ast}_n$. This in turn
implies by Lemma \ref{ordertype}, that $\tmop{ot} (C_{s_n^{\ast}})$
is less than the least p.r. closed ordinal greater than $\beta'$.
However $\pi^{\ast}_n \upharpoonright \beta_n^+$ extends $\pi^{}
\upharpoonright \beta_n^+$, and thus this ordinal is $\pi^{} (
\tilde{\beta}) .$ The final inequality is clear.\\

Now (10) yields the final contradiction, as for relevant $n$, $S_n$
was chosen to consist of points $\beta$ where $\tmop{ot} (C_{\beta})
\geq \omega_{n - 3}$, whereas (10) establishes an ultimate bound on
such order types of $\omega_{k_0 + 1}$. \mbox{}\hfill Q.E.D.(Theorem
\ref{0-sharp})\\

We finally remark that the Corollary 1.5 is immediate: after
shifting our attention to cardinals above $\aleph_k$ we still use
the same hypothesis concerning sufficient singular ordinals in $K$
in order to establish the stationarity of the $T_n$ now contained in
$\mathrm{Cof}(\omega_k)$. We take $\omega_k \subseteq X$  and now
the analogues of the ordinals $\beta_n$ have cofinality $\omega_k$;
  $H$ is correct
about the cofinality of any ordinal whose $V$-cofinality is less
than $\omega_k$. The proof of (3) now shows that there is no closed
$\omega_k$ subsequence of critical points $\kappa_i$ unbounded in
such a $\beta_n$, as the map $\pi$ is now continuous at points of
cofinality less than $\omega_k$. Hence we can deduce (4) that the
iterates are indeed singularizing structures for the $\beta_n$ as
required.

\end{document}